\documentclass[preprint]{elsarticle}
\usepackage [english]{babel}
\usepackage[utf8]{inputenc}
\usepackage{amssymb}
\usepackage{amsbsy}
\usepackage{amsmath,stmaryrd}
\usepackage{amsfonts}
\usepackage{booktabs}
\usepackage{graphicx}
\usepackage{url}
\usepackage{algorithm}
\usepackage{parskip}
\usepackage{makecell}
\usepackage{color}
\usepackage{tikz}
\usepackage{pgfplots}
\pgfplotsset{compat=1.7}
\usepackage{xifthen}
\usepackage{hyperref}
\journal{arXiv}
\newcommand{\ol}{\overline}
\newcommand{\ul}{\underline}
\newcommand{\mc}{\mathcal}
\newcommand{\mcH}{\mc{H}}
\newcommand{\mcN}{\mc{N}}
\newcommand{\ee}{\textbf{e}}

\newcommand{\nn}{\textnormal{\textbf{n}}}
\newcommand{\dd}{\textnormal{d}}

\newtheorem{theorem}{Theorem}
\newtheorem{lemma}[theorem]{Lemma}

\newtheorem{definition}[theorem]{Definition}
\newtheorem{assumption}[theorem]{Assumption}

\newdefinition{remark}{Remark}
\newproof{proof}{Proof}

\begin{document}

\begin{frontmatter}

\title{Discretization error estimates for discontinuous Galerkin Isogeometric Analysis}

%Error estimates for discontinuous Galerkin IgA

\author[ricam]{Stefan Takacs}
\ead{stefan.takacs@ricam.oeaw.ac.at}

\address[ricam]{Johann Radon Institute for Computational
and Applied Mathematics (RICAM),\\
Austrian Academy of Sciences}

\begin{abstract}
Isogeometric Analysis is a spline-based discretization method to partial differential equations which shows the approximation power of a high-order method. The number of degrees of freedom, however, is as small as the number of degrees of freedom of a low-order method. This does not come for free as the original formulation of Isogeometric Analysis requires a global geometry function. Since this is too restrictive for many kinds of applications, the domain is usually decomposed into patches, where each patch is parameterized with its own geometry function. In simpler cases, the patches can be combined in a conforming way. However, for non-matching discretizations or for varying coefficients, a non-conforming discretization is desired. An symmetric interior penalty discontinuous Galerkin (SIPG) method for Isogeometric Analysis has been previously introduced. In the present paper, we give error estimates that are explicit in the spline degree. This opens the door towards the construction and the analysis of fast linear solvers, particularly multigrid solvers for non-conforming multipatch Isogeometric Analysis.
\end{abstract}


\end{frontmatter}

\section{Introduction}

The original design goal of Isogeometric Analysis (IgA),~\cite{Hughes:2005}, was
to unite the world of computer aided design (CAD) and the world of finite element
(FEM) simulation. In IgA, both the computational domain and the solution of the
partial differential equation (PDE) are represented by spline functions,
like tensor product B-splines or non-uniform rational B-splines (NURBS). This follows
the design goal since such spline functions are also used in standard CAD
systems to represent the geometric objects of interest.

The parameterization of the computational domain using just \emph{one} tensor-product
spline function, is possible only in simple cases. A necessary condition for this to
be possible, is that the computational domain is topologically equivalent to the unit
square or the unit cube. This might not be the case for more complicated computational domains.
Such domains are typically decomposed into subdomains, in IgA called patches, where
each of them is parameterized by its own geometry function. The standard approach
is to set up a conforming discretization. For a standard Poisson problem, this
means that the overall discretization needs to be continuous. For higher order problems,
like the biharmonic problem, even more regularity is required; conforming discretizations
in this case are rather hard to construct, cf.~\cite{Kapl:Sangalli:Takacs:2018} and references therein.

Even for the Poisson problem, a conforming discretization requires the
discretizations to agree on the interfaces.
This excludes many cases of practical interest, like having different grid sizes or
different spline degrees on the patches. Since such cases might be of interest, alternatives
to conforming discretizations are of interest. One promising alternative are discontinuous
Galerkin approaches, cf.~\cite{Riviere:2008,Arnold:Brezzi:Cockburn:Marini:2002}, particularly
the symmetric interior penalty discontinuous Galerkin (SIPG) method~\cite{Arnold:1982}. 
The idea of applying this technique to couple patches in IgA, has been previously discussed
in~\cite{LMMT:2015,LT:2015}.

Concerning the approximation error, in early IgA literature, only its
dependence on the grid size has been studied, cf.~\cite{Hughes:2005,Bazilevs:BeiraoDaVeiga:Cottrell:Hughes:Sangalli:2006}. In recent
publications~\cite{BeiraoDaVeiga:Buffa:Rivas:Sangalli:2011,Takacs:Takacs:2015,Floater:Sande:2017,Sande:Manni:Speleers:2018} also the
dependence on the spline degree has been investigated.
These error estimates are restricted to the single-patch case.
In~\cite{Takacs:2018}, the results from~\cite{Takacs:Takacs:2015} on approximation errors
for B-splines of maximum smoothness have been extended to the conforming multi-patch case.

For the case of discontinuous Galerkin discretizations, only error estimates in the grid
size are known, cf.~\cite{LT:2015}. The goal of the present paper is to present an analysis that is explicit both in the grid size $h$ and the spline degree $p$. We show that the penalty parameter 
has to grow like $p^2$ for the SIPG method to be well-posed.
If the solution is sufficiently smooth, the error in the energy norm decays like $h^{p}$. For the analysis of linear solvers, we need also estimates for less smooth functions. We give a bound that degrades only poly-logarithmically with $p$, cf.~\eqref{eq:starstar}.
These error estimates can be used to analyze multigrid solvers~\cite{Takacs:2019a} and FETI-like solvers~\cite{SchneckenleitnerTakacs:2020} for discontinuous Galerkin multipatch discretizations.

The remainder of the paper is organized as follows. In Section~\ref{sec:prelim}, we introduce
the model problem and give a detailed description of its discretization. A discussion of
the existence of a unique solution and the discretization and the approximation
error, is provided in Section~\ref{sec:discrerror}. The proof of the approximation error estimate
is given in Section~\ref{sec:approx:proof}. We provide numerical experiments that depict
our estimates, in Section~\ref{sec:num}. 

\section{The model problem and its discretization}\label{sec:prelim}
 
We consider the following \emph{Poisson model problem}.
Let $\Omega\subset\mathbb{R}^2$ be an open and simply
connected Lipschitz domain.
For any given source function $f\in L_2(\Omega)$,
we are interested in the function $u\in H^{1,\circ}(\Omega):=H^1(\Omega)\cap L_2^\circ(\Omega)$ solving
\begin{equation} \label{eq:model}
		(  \nabla u,\nabla v)_{L_2(\Omega)} = (f,v)_{L_2(\Omega)}
			\qquad \mbox{for all $v \in H^{1,\circ}(\Omega)$.}
\end{equation}
Here and in what follows, for any $r\in \mathbb{N}:=\{1,2,3,\ldots\}$,
$L_2(\Omega)$ and $H^r(\Omega)$
are the standard Lebesgue and Sobolev spaces with standard scalar products
$(\cdot,\cdot)_{L_2(\Omega)}$, $(\cdot,\cdot)_{H^r(\Omega)}:=(\nabla^r\cdot,\nabla^r\cdot)_{L_2(\Omega)}$, 
norms $\|\cdot\|_{L_2(\Omega)}$ and $\|\cdot\|_{H^r(\Omega)}$, and seminorms $|\cdot|_{H^r(\Omega)}$. 
The Lebesgue space of function with zero mean is given by $L_2^\circ(\Omega):=\{v\in L_2(\Omega)\;:\;(v,1)_{L_2(\Omega)}=0\}$.

The computational domain $\Omega$ is the union of $K$ non-overlapping open patches $\Omega_k$, i.e.,
\begin{equation} \label{eq:matching}
	\ol{\Omega} = \bigcup_{k=1}^K \ol{\Omega_k}
	\quad \mbox{and} \quad
	\Omega_k\cap\Omega_l=\emptyset \mbox{ for any }k\not=l
\end{equation}
holds, where $\ol{T}$ denotes the closure of $T$. 
Each patch $\Omega_k$ is represented by a bijective geometry function
\[
		G_k :\widehat{\Omega}:=(0,1)^2 \rightarrow \Omega_k := G_k (\widehat{\Omega})\subset \mathbb{R}^2,
\]
which can be continuously extended to the closure of $\widehat{\Omega}$ such that
$G_k(\ol{\widehat{\Omega}})=\ol{\Omega_k}$.
We use the notation
\[
		v_k := v|_{\Omega_k}
		\qquad \mbox{and} \qquad
		\widehat{v}_k := v_k \circ G_k
\]
for any function $v$ on $\Omega$. If $v\in H^1(\Omega)$, we can use standard trace theorems
to extend $v_k$ to $\ol{\Omega_k}$ and to extend $\widehat{v}_k$ to $\ol{\widehat{\Omega}}$.

We assume that the mesh induced by the interfaces between the patches does not have
any T-junctions, i.e., we assume as follows.
\begin{assumption}\label{ass:1}
	For any two patches $\Omega_k$ and $\Omega_l$ with $k\not=l$, 
	the intersection $\partial\Omega_k \cap \partial\Omega_l$ is either
	(a) empty, (b) a common vertex, or (c) a common edge $I_{k,l}=I_{l,k}$ such
	that
	\begin{equation}\label{eq:b:seq}
				\widehat{I}_{k,l}:=I_{k,l}\circ G_k \in \widehat{\mc{I}} := \{
							\{0\}\times(0,1), \{1\}\times(0,1), (0,1)\times\{0\}, (0,1)\times\{1\}
		\}.
	\end{equation}
\end{assumption}
Note that the pre-images $\widehat{I}_{k,l}$ and $\widehat{I}_{l,k}$ do not necessarily agree. 
We define
\begin{align*}
		\mcN(k)&:= \{ l \in \{1,\ldots,K\} \;:\; \Omega_k \mbox{ and } \Omega_l \mbox{ have
		common edge} \},
\end{align*}
$
		\mcN:= \{ (k,l) \; : \; k<l \mbox{ and } l\in \mcN(k) \} ,
$ and
$		\mcN^*:= \{ (k,l) \; : \; k>l \mbox{ and } l\in \mcN(k) \} ,
$

and the parameterization
$\gamma_{k,l}: (0,1)\rightarrow \widehat{I}_{k,l}$ via
\begin{equation}\label{eq:gamma:kl:def}
	\gamma_{k,l}(t) := \left\{
			\begin{array}{cc}
					(t,s) & \quad \mbox{if} \quad \widehat{I}_{k,l} = (0,1) \times \{s\}, \quad	
											s \in \{0,1\} \\
					(s,t) & \quad \mbox{if} \quad \widehat{I}_{k,l} = \{s\} \times (0,1), \quad	
											s \in \{0,1\}.
			\end{array}
		\right..
\end{equation}
We assume that the geometry functions agree on the interface (up to the
orientation); this does not require any smoothness of the overall geometry function
normal to the interface.

\begin{assumption}\label{ass:geoequiv:bdy}
		For all $(k,l) \in \mathcal{N}^*$ and $t\in (0,1)$, we have 
		\[
							\gamma_{k,l}(t) = G_k^{-1} \circ G_l \circ \gamma_{l,k}(t) 
							\quad \mbox{or} \quad
							\gamma_{k,l}(t) = G_k^{-1} \circ G_l \circ \gamma_{l,k}(1-t).
		\] 
\end{assumption}
\begin{remark}
	For any domain satisfying Assumption~\ref{ass:1}, we
	can reparameterize each patch such that this condition is satisfied. Assume
	to have two patches $\Omega_k$ and $\Omega_l$, sharing the patch
	$I_{k,l} = G_k((0,1)\times\{0\}) = G_l((0,1)\times\{0\})$. Using
	\[
		\widetilde{G}_k(x,y)
			:=
			G_k( y x + (1-y) \rho(x)   , y)
	\]
	where
	\[
		(\rho(t),0) := \left\{
			\begin{array}{ll}
				G_k^{-1}\circ G_l(t,0) & \quad\mbox{if } G_k(s,0)=G_l(s,0) \mbox{ for } s\in\{0,1\} \\
				G_k^{-1}\circ G_l(1-t,0) & \quad\mbox{if } G_k(s,0)=G_l(1-s,0) \mbox{ for } s\in\{0,1\} \\
			\end{array}
			\right.,
	\]
	we obtain a reparameterization of $G_k$, which (a) matches the parameterization
	of $\Omega_l$ at the interface, (b) is unchanged on the other interfaces, and (c)
	keeps the patch $\Omega_k$ unchanged. By iteratively applying this approach to all
	patches, we obtain a discretization satisfying Assumption~\ref{ass:geoequiv:bdy}.
\end{remark}

We assume that the geometry function is
sufficiently smooth such that the following assumption holds.
\begin{assumption}\label{ass:geoequiv}
		There is a constant $C_G>0$ such that the geometry functions $G_k$ satisfy
		the estimates
		\begin{equation}\label{eq:ass3}
			\sup_{x\in \widehat{\Omega}}
				\| \nabla^r G_k(x)\|_{\ell_2}  \le C_G
				\quad \mbox{and}\quad
				\sup_{x\in \widehat{\Omega}}
				\| (\nabla^r G_k(x))^{-1} \|_{\ell_2}  \le C_G
		\end{equation}
		for $r= 1,2$.
\end{assumption}

We assume full elliptic regularity.
\begin{assumption}\label{ass:reg}
	The solution $u$ of the model problem~\eqref{eq:model}
	satisfies
	\[
			u_k \in H^{2}(\Omega_k)
	\]
	for all $k=1,\ldots,K$.
\end{assumption}
For domains $\Omega$
with a sufficiently smooth boundary, cf.~\cite{Necas:1967}, and for convex
polygonal domains $\Omega$, cf.~\cite{Dauge:1988,Dauge:1992},
we have $u\in H^2(\Omega)$ and
thus also Assumption~\ref{ass:reg}. In case of varying diffusion
conditions (which are uniform on each patch), we might have
$u\not \in H^2(\Omega)$, but Assumption~\ref{ass:reg} might
still be satisfied cf.~\cite{Petzoldt:2001,Petzoldt:2002} and others.
The theory of this paper can be extended to cases where we only
know $u_k \in H^{3/2+\epsilon}(\Omega_k)$ for some $\epsilon>0$. For simplicity, we
restrict ourselves to the case of full elliptic regularity, i.e.,
Assumption~\ref{ass:reg}.

Having a representation of the domain, we introduce the isogeometric function space.
Following \cite{LMMT:2015,LT:2015}, we use a conforming isogeometric discretization for
each patch and couple the contributions for the patches using a symmetric interior
penalty discontinuous Galerkin (SIPG) method, cf.~\cite{Arnold:1982}, as follows.

For the \emph{univariate case}, the space of spline functions
of degree $p\in \mathbb{N}$
over a grid (vector
of breakpoints) $Z:=(\zeta_0,\zeta_1,\ldots,\zeta_{n-1},\zeta_n)$
with $\zeta=0$ and $\zeta_n=1$
and size $h:=\max_{i=0,\ldots,n-1} |\zeta_{i+1}-\zeta_i|$ is given by 
	\begin{equation*}
		S_{p,\Xi}(0,1) := \left\{ v \in C^{p-1}(0,1): \; v |_{(\zeta_i,\zeta_{i+1}]} \in \mathbb{P}^p \mbox{ for all } i=1,\ldots,n-1 \right\},
	\end{equation*}
where $\mathbb{P}^p$ is the space of polynomials of degree $p$.

On the \emph{parameter domain} $\widehat{\Omega}:=(0,1)^2$, we introduce
tensor-product B-spline functions
\begin{equation*}
		\widehat V_k := S_{p_k,Z_{k,1}}(0,1) \otimes S_{p_k,Z_{k,2}}(0,1).
\end{equation*}

The \emph{multi-patch function space} $V_h$  is given by
\begin{equation}\label{eq:igaspace}
	V_h :=
	\{
	u_h \in L_2^\circ(\Omega)
		\; : \;
		u_h \circ G_k \in \widehat V_k \mbox{ for } k=1,\ldots,K,
	\} .
\end{equation}
Note that the grid sizes $h_k$ and the spline degrees $p_k$ can be different for each of
the patches. We define
\[
		p:= \max_{k\in\{1,\ldots,K\}} p_k, 
		\qquad p_{\min} := \min_{k\in\{1,\ldots,K\}} p_k ,
		\qquad  h:=\max_{\delta\in\{1,2\}}\max_{k\in\{1,\ldots,K\}} h_{k,\delta}
\]
to be the largest spline degree, the smallest spline degree and the
grid size. We assume 
\begin{equation}\label{eq:p2}
		p_{\min} \ge 2 \quad \mbox{and} \quad
		\min_{\delta\in\{1,2\}}
		\min_{k\in\{1,\ldots,K\}} h_{k,\delta,\min} \ge C_h h
\end{equation}
for some constant $C_h>0$,
where $h_{k,\delta,\min}$ refers to the smallest knot span.

Following the assumption that $u_h$ is a patchwise function, we define for each $r\in\mathbb{N}$
a broken Sobolev space
\[
		\mcH^r(\Omega) := \{ v\in L_2(\Omega) \;:\; v_k \in H^r(\Omega)\},
\]
with associated norms and scalar products
\[		
		\|v\|_{\mcH^r(\Omega)} := (v,v)_{\mcH^r(\Omega)}^{1/2} \quad
		\mbox{and}\quad
		(u,v)_{\mcH^r(\Omega)}:=\sum_{k=1}^K( u, v)_{H^r(\Omega_k)}.
\]

For each patch, we define on its boundary $\partial\Omega_k$ the outer normal vector~$\nn_k$.
On each interface $I_{k,l}$, we define the jump operator $\llbracket \cdot \rrbracket$ by
\[
		\llbracket v \rrbracket := v_k - v_l \qquad\mbox{on } I_{k,l}=I_{l,k} \mbox{ where } (k,l) \in \mcN
\]
and the average operator $\{ \cdot \}$ by
\[
		\{ v \} := \tfrac12 (v_k + v_l) \qquad\mbox{on } I_{k,l}=I_{l,k} \mbox{ where } (k,l) \in \mcN.
\]

The discretization of the variational problem using the \emph{symmetric interior penalty discontinuous
Galerkin method} reads as follows. Find $u_h\in V_h$ such that
\begin{equation} \label{eq:model:discr}
		(u_h,v_h)_{A_h} = (f,v_h)_{L_2(\Omega)}
			\qquad \mbox{for all }v_h \in V_h,
\end{equation}
where
\begin{align*}
	(u,v)_{A_h} &:=	(u,v)_{\mcH^1(\Omega)} - (u,v)_{B_h} - (v,u)_{B_h} + (u,v)_{C_h},\\
	(u,v)_{B_h} &:=  \sum_{(k,l)\in \mcN} (\llbracket u \rrbracket,\{ \nabla v\}\cdot \nn_k)_{L_2(I_{k,l})} ,
	 \\
	(u,v)_{C_h} &:=\frac{\sigma}{h} \sum_{(k,l)\in \mcN}   (\llbracket u\rrbracket,\llbracket v\rrbracket)_{L_2(I_{k,l})}
\end{align*}
for all $u,v \in \mcH^{2,\circ}(\Omega)$, where the
penalty parameter 
\begin{equation}\label{eq:6a}
		\sigma = \sigma_0 p^2>0
\end{equation}
is chosen sufficiently large.

Using a basis for the space $V_h$, we obtain a standard matrix-vector problem:
Find $\ul{u}_h \in \mathbb{R}^N$ such that
\begin{equation} \label{eq:linear:system}
			A_h \ul{u}_h = \ul{f}_h.
\end{equation}
Here and in what follows,
$\ul{u}_h=[u_i]_{i=1}^N$ is the coefficient vector representing $u_h$ with respect
to the chosen basis, i.e., $u_h=\sum_{i=1}^N u_i \varphi_i$, and $\ul{f}_h =  [ (f, \varphi_i)_{L_2(\Omega)} ]_{i=1}^N$ is the coefficient
vector obtained by testing the right-hand-side functional with
the basis functions.

As the dependence on the geometry function is not in the focus of this paper,
unspecified constants might depend on $C_G$, $C_I$ and $C_h$.
Before we proceed, we introduce a convenient notation. 
\begin{definition}
	Any generic constant $c>0$ used within this paper is understood to be 
	independent of the grid size $h$, the spline degree $p$ and the number of patches $K$,
	but it might depend on the constants $C_G$, $C_I$ and $C_h$.

	We use the notation $a\lesssim b$ if there is a generic constant $c>0$ such that $a\le c b$ and
	the notation $a \eqsim b$ if $a\lesssim b$ and $b\lesssim a$.

	For symmetric positive definite matrices $A$ and $B$, we write
	\[
		A \le B \qquad \mbox{if} \qquad 
		\ul{v}_h^{\top} A \ul{v}_h \le\ul{v}_h^{\top} B \ul{v}_h \quad \mbox{ for all vectors } \ul{v}_h.
	\]
	The notations $A\lesssim B$ and $A \eqsim B$ are defined analogously.
\end{definition}

\section{A discretization error estimate}\label{sec:discrerror}

In \cite{LMMT:2015}, it has been shown that the bilinear form $(\cdot,\cdot)_{A_h}$ is coercive
and bounded in the dG-norm. For our further analysis, it is vital to know these conditions
to be satisfied with constants that are independent of the spline degree $p$. Thus, we define
the dG-norm via
\begin{align*}
	\|u\|_{Q_h}^2:= (u,u)_{Q_h}, \qquad \mbox{where}\qquad
		(u,v)_{Q_h} :=
		(u,v)_{\mcH^1(\Omega)}
		+ (u,v)_{C_h} 
\end{align*}
for all $u,v \in \mcH^{2,\circ}(\Omega)$.
Note that we define the norm differently to~\cite{LMMT:2015}, where the dG-norm was 
independent of $p$.

Before we proceed, we give some estimates on the geometry functions.
\begin{lemma}\label{lem:geoequiv:1}
	We have
	\begin{equation}\label{eq:lem:geoequiv:1}
		 \|v\circ G_k^{-1}\|_{H^r(\Omega_k)}  \eqsim \|v\|_{H^r(\widehat{\Omega})}
		\quad \mbox{ for all } \quad v \in H^r(\widehat{\Omega}), 
	\end{equation}
	where $r=0,1,2$. 
	For ease of notation, here and in what follows, we define $H^0:=L_2$.
	If~\eqref{eq:ass3} holds for $r=1,2,\ldots,s$
	with some $s>r$, then~\eqref{eq:lem:geoequiv:1} also holds
	for those choices of $r$.	Moreover, we have
	\[	
		   \|v\circ G_k^{-1}\|_{L_2(I_{k,l})}  \eqsim \|v\|_{L_2(\widehat{I}_{k,l})}
		\quad \mbox{ for all } \quad v \in H^1(\widehat{\Omega}).
	\]
\end{lemma}
\begin{proof}
	The statements follow directly from the chain rule for differentiation,
	the substitution rule for integration and Assumption~\ref{ass:geoequiv}.
\qed\end{proof}

\begin{lemma}\label{lem:geoequiv:2}
	We have
	$
		  \|(\nabla v\circ G_k^{-1})\cdot \nn_k \|_{L_2(I_{k,l})} \lesssim \|\nabla v\|_{L_2(\widehat{I}_{k,l})}
	$
	for all $v \in H^2(\widehat{\Omega})$.
\end{lemma}
\begin{proof}
	We have 
	$
		\|(\nabla v\circ G_k^{-1})\cdot \nn_k\|_{L_2(I_{k,l})}
		\le
		\| \nabla v\circ G_k^{-1}  \|_{L_2(I_{k,l})} \|\nn_k\|_{L_\infty(I_{k,l})},
	$
	where certainly $\|\nn_k\|_{L_\infty(I_{k,l})}=1$ because the length of $\nn_k$ is always $1$.
	The estimate $\| \nabla v\circ G_k^{-1}  \|_{L_2(I_{k,l})}\lesssim 
	\|\nabla v\|_{L_2(\widehat{I}_{k,l})}$ follows directly from the chain
	rule for differentiation, the substitution rule for integration
	and Assumption~\ref{ass:geoequiv}.
\qed\end{proof}

For $\sigma$ sufficiently large, the symmetric
bilinear form $(\cdot,\cdot)_{A_h}$ is coercive and bounded, i.e.,
a scalar product.

\begin{theorem}[Coercivity and boundedness]\label{thrm:1}
	There is some $\sigma_0>0$ that only depends on $C_G$, $C_I$,
	and $C_h$ such that
	\[
		(u_h,u_h)_{A_h} \gtrsim \|u_h\|_{Q_h}^2\quad\mbox{and}\quad
		(u_h,v_h)_{A_h} \lesssim \|u_h\|_{Q_h} \| v_h\|_{Q_h}
	\]
	holds for all $u_h,v_h \in  V_h$ and all $\sigma\ge p^2 \sigma_0$.
\end{theorem}
\begin{proof}
	Note that
	$
		(u_h,v_h)_{A_h} = (u_h,v_h)_{Q_h} 
					- (u_h,v_h)_{B_h} - (v_h,u_h)_{B_h}.
	$
	Using Lemma~\ref{lem:geoequiv:2}, \cite[Lemma~4.4]{Takacs:2018},
	\cite[Corollary~3.94]{Schwab:1998} and Lemma~\ref{lem:geoequiv:1}, we obtain
	\begin{align}
		&\|\nabla v_h\cdot \nn_k\|_{L_2(I_{k,l})}^2 \lesssim
		\|\nabla (v_h\circ G_k)\cdot \nn_k\|_{L_2(\widehat{I}_{k,l})}^2 \lesssim
		\|v_h\circ G_k\|_{H^{2}(\widehat{\Omega})}\|v_h\circ G_k\|_{H^{1}(\widehat{\Omega})} \nonumber\\&\lesssim
		\frac{p^2}h\|v_h\circ G_k\|_{H^{1}(\widehat{\Omega})}^2 \lesssim
		\frac{p^2}h\|v_h\|_{H^{1}(\Omega_k)}^2\label{eq:theinverse}
	\end{align}
	for all $v_h \in V_h$, $k=1,\ldots,K$ and $l \in \mcN(k)$.
	As $V_h \subset L_2^\circ(\Omega)$, the Poincar\'e inequality
	(see, e.g.,~\cite[Theorem~A.25]{Schwab:1998}) yields also
	\[
			\|\nabla v_h\cdot \nn_k\|_{L_2(I_{k,l})}^2 \lesssim\frac{p^2}h|v_h|_{H^{1}(\Omega_k)}^2.
	\]
	The Cauchy-Schwarz inequality, the triangle
	inequality, \eqref{eq:theinverse}, 
	and $|\mcN(k)|\le 4$ yield
	\begin{equation}\label{eq:b:bound}
	\begin{aligned}
		&|(u_h,v_h)_{B_h}| \le 
			\left(\sum_{(k,l)\in \mcN} 
					\|\llbracket u_h \rrbracket\|_{L_2(I_{k,l})}^2 
			\right)^{1/2}\left(\sum_{(k,l)\in \mcN} 
				\|\{\nabla v_h\}\cdot \nn_k\|_{L_2(I_{k,l})}^2
			\right)^{1/2}
			 \\
			&\quad\lesssim 
			\left(\sum_{(k,l)\in \mcN} 
					\|\llbracket u_h \rrbracket\|_{L_2(I_{k,l})}^2 
			\right)^{1/2}\left(\sum_{k=1}^K \sum_{j\in\mcN(k)}
				\|\nabla v_h\cdot \nn_k\|_{L_2(I_{k,l})}^2
			\right)^{1/2} 
			\\
			&\quad\lesssim \left(\frac{p^2}{h}\right)^{1/2}
			\left(\sum_{(k,l)\in \mcN} 
					\|\llbracket u_h \rrbracket\|_{L_2(I_{k,l})}^2 
			\right)^{1/2}
			\left(\sum_{k=1}^K 
				|v_h|_{H^{1}(\Omega_k)}^2
			\right)^{1/2}
				 \\
			&\quad\le p\,\sigma^{-1/2}
			 \| u_h \|_{Q_h}
			\| v_h \|_{Q_h}
	\end{aligned}
	\end{equation}
	for all $u_h,v_h\in V_h$. Let $c_0\eqsim 1$ be the hidden constant, i.e., such that
	\begin{equation}\label{eq:c0:def}
			|(u_h,v_h)_{B_h}| \le c_0 \,p\, \sigma^{-1/2} \| u_h \|_{Q_h}
			\| v_h \|_{Q_h}.
	\end{equation}
	For $\sigma \ge 16 \,c_0\,p^2$, we obtain
		\[
		(u_h,u_h)_{A_h} = \|u_h\|_{Q_h}^2 - 2 (u_h,u_h)_{B_h}\ge 
			 \frac12 \|u_h\|_{Q_h}^2,
	\]
	i.e., coercivity. Using~\eqref{eq:b:bound} and the Cauchy-Schwarz inequality,
	we obtain further
	\[
		(u_h,v_h)_{A_h} = (u_h,v_h)_{Q_h} -  (u_h,v_h)_{B_h} -  (v_h,u_h)_{B_h}\le 
			 \frac32 \|u_h\|_{Q_h}\|v_h\|_{Q_h},
	\]
	i.e.,
	boundedness.
\qed\end{proof}

As we have boundedness and coercivity (Theorem~\ref{thrm:1}), the Lax Milgram
theorem (see, e.g.,~\cite[Theorem~1.24]{Schwab:1998}) yields
states existence and uniqueness of a solution, i.e., the following statement.
\begin{theorem}[Existence and uniqueness]
	If $\sigma$ is chosen as in Theorem~\ref{thrm:1}, the problem~\eqref{eq:model:discr}
	has exactly one solution $u_h\in V_h$.
\end{theorem}

The following theorem shows that
the solution of the original problem also satisfies the discretized bilinear form.
\begin{theorem}[Consistency]\label{thrm:consit}
	The solution $u\in H^{1,\circ}(\Omega)\cap \mc H^2(\Omega)$
	of the original problem~\eqref{eq:model}
	satisfies
	\[
			(u,v_h)_{A_h} = (f,v_h)_{L_2(\Omega)} \quad \mbox{for all } v_h \in V_h.
	\]
\end{theorem}
For a proof, see, e.g.,~\cite[Proposition~2.9]{Riviere:2008}; the proof requires
elliptic regularity (cf. Assumption~\ref{ass:reg}).

If boundedness of the bilinear form  $(\cdot,\cdot)_{A_h}$ was also
satisfied for $u\in  \mcH^{2,\circ}(\Omega)$,
Ce\'a's Lemma (see, e.g.,~\cite[Theorem~2.19.iii]{Schwab:1998})
would allow to bound the discretization error. However, the
bilinear form is not bounded in the norm
$\|\cdot\|_{Q_h}$, but only in the stronger norm $\|\cdot\|_{Q_h^+}$,
given by
	\begin{equation}\label{eq:def:qhplus}
		\|u\|_{Q_h^+}^2:=
		\|u\|_{Q_h}^2+
				\frac{h^2}{\sigma^2} |u|_{\mc H^2(\Omega)}^2.
	\end{equation}
\begin{theorem}\label{thrm:1a}
	If $\sigma$ is chosen as in Theorem~\ref{thrm:1},
	\[
		(u,v_h)_{A_h} \lesssim \|u\|_{Q_h^+} \|v_h\|_{Q_h}
	\]
	holds for all $u \in \mcH^{2,\circ}(\Omega)$ and all
	$v_h \in V_h$.
\end{theorem}
\begin{proof}
	Let $u\in \mcH^{2,\circ}(\Omega)$ and $v_h\in V_h$ be arbitrarily but fixed.
	Note that the arguments from~\eqref{eq:b:bound} also hold if the first
	parameter of the bilinear form $(\cdot,\cdot)_{B_h}$ is not in $V_h$. So, we obtain
	\[
			|(u,v_h)_{B_h}| \le  \frac{1}{4} \|u\|_{Q_h} \|v_h\|_{Q_h}.
	\]
	Using Lemma~\ref{lem:geoequiv:2}, \cite[Lemma~4.4]{Takacs:2018},
	Lemma~\ref{lem:geoequiv:1} and the Poincar\'e inequality, we obtain
	\begin{align}
		&\|\nabla v\cdot \nn_k\|_{L_2(I_{k,l})}^2 \lesssim
		\|\nabla (v\circ G_k)\cdot \nn_k\|_{L_2(\widehat{I}_{k,l})}^2 \lesssim
		\|v\circ G_k\|_{H^{2}(\widehat{\Omega})}\|v\circ G_k\|_{H^{1}(\widehat{\Omega})} \nonumber\\&\lesssim
		\|v\|_{H^{2}(\Omega_k)}\|v\|_{H^{1}(\Omega_k)}
		\le
		\frac{1}{\beta} \|v\|_{H^{2}(\Omega_k)}^2+\beta \|v\|_{H^{1}(\Omega_k)}^2
		\le
		\frac{1}{\beta} |v|_{H^{2}(\Omega_k)}^2+\beta |v|_{H^{1}(\Omega_k)}^2
		\label{eq:theinverse2}
	\end{align}
	for all $v \in H^2(\Omega_k)$, all $k=1,\ldots,K$, all $l \in \mcN(k)$ and all $\beta>1$.
	Using this estimate,
	and $|\mcN(k)|\le 4$, we obtain for $\beta:=h^{-2}\sigma$
	\begin{align*}
		&|(v_h,u)_{B_h}| 
			 \le 
			\left(\frac{\sigma}{h} \sum_{(k,l)\in \mcN}
					\|\llbracket v_h \rrbracket\|_{L_2(I_{k,l})}^2 
			\right)^{1/2}
			\left(\frac{h}{\sigma}
				\sum_{k=1}^K \sum_{l\in \mcN(k)} 
				\|\nabla u\cdot \nn_k\|_{L_2(I_{k,l})}^2
			\right)^{1/2} 
			\\
			&\quad \le 
			\|v_h\|_{Q_h}
			\left(\sum_{k=1}^K 
				 |u|_{H^1(\Omega_k)}^2
				+ \frac{h^2}{\sigma^2} \sum_{k=1}^K  |u|_{H^2(\Omega_k)}^2
			\right)^{1/2} 
			 \lesssim  \|v_h\|_{Q_h} \|u\|_{Q_h^+}.
	\end{align*}
	Using these estimates, we obtain
	\begin{align*}
			(u,v_h)_{A_h} &= (u,v_h)_{Q_h}-(u,v_h)_{B_h}-(v_h,u)_{B_h} 
				\lesssim \|u\|_{Q_h^+} \|v_h\|_{Q_h},
	\end{align*}
	which finishes the proof.
\qed\end{proof}

Using consistency (Theorem~\ref{thrm:consit}), coercivity and boundedness
(Theorems~\ref{thrm:1} and~\ref{thrm:1a}), we 
can bound the discretization error using a the approximation error.
\begin{theorem}[Discretization error estimate]\label{thrm:cea}
	Provided the assumptions of
	Theorems~\ref{thrm:1} and~\ref{thrm:consit}, the
	estimate
	\[
		\|u-u_h\|_{Q_h} \lesssim \inf_{v_h \in V_h} \|u-v_h\|_{Q_h^+}
	\]
	holds,
	where $u$ is the solution of the original problem~\eqref{eq:model}
	and $u_h$ is the
	solution of the discrete problem~\eqref{eq:model:discr}.
\end{theorem}
\begin{proof}
	For any $v_h\in V_h$, the triangle inequality yields
	\begin{equation} \label{eq:thrm:cea}
		\|u-u_h\|_{Q_h} \le 
					\|u-v_h\|_{Q_h} + \|v_h-u_h\|_{Q_h}.
	\end{equation}
	Theorem~\ref{thrm:consit} and Galerkin orthogonality yield $(u-u_h,w_h)_{A_h}=0$ for
	all $w_h\in V_h$.
	So, we obtain using Theorems~\ref{thrm:1} and~\ref{thrm:1a} that
	\[
		\|v_h-u_h\|_{Q_h}^2 \lesssim
				(v_h-u_h,v_h-u_h)_{A_h}=
				(v_h-u,v_h-u_h)_{A_h}
					\lesssim \|v_h-u\|_{Q_h^+} \|v_h-u_h\|_{Q_h},
	\]
	which shows $\|v_h-u_h\|_{Q_h}\lesssim \|u-v_h\|_{Q_h^+}$. Together with~\eqref{eq:thrm:cea},
	this shows
	$\|u-u_h\|_{Q_h} \lesssim \|u-v_h\|_{Q_h^+}$. Since this holds for all $v_h\in V_h$,
	this finishes the proof.
\qed\end{proof}

\begin{theorem}[Approximation error estimate]\label{thrm:high}
	Let $q \in \{1,\ldots,p_{\min}\}$. Provided that
	$\sigma$ is as in Theorem~\ref{thrm:1} and that
	$\|u_k\|_{H^{q+1}(\Omega_k)}\eqsim
	\|\widehat u_k\|_{H^{q+1}(\widehat \Omega)}$ for
	$k=1,\ldots,K$
	(cf. Lemma~\ref{lem:geoequiv:1}), then
	\begin{equation}\label{eq:thrm:high}
		\inf_{v_h \in V_h} \|u-v_h\|_{Q_h^+} \lesssim
		\sigma^{1/2}  \, \pi^{-q} \, h^{q} 
				\|u\|_{\mcH^{q+1}(\Omega)}
	\end{equation}
   holds for all $u\in H^{1,\circ}(\Omega)\cap \mc H^{q+1}(\Omega)$.
\end{theorem}
A proof of this theorem is given at the end of the next section.

Assuming $p\eqsim p_{\min}$, then we have for the case $q=p_{\min}$
that
\begin{equation}\nonumber
	\|u-u_h\|_{Q_h} \lesssim 
		\inf_{v_h \in V_h} \|u-v_h\|_{Q_h^+} \lesssim  
		\sigma_0^{1/2} \, \underbrace{q^2  \, \pi^{-q}}_{\displaystyle < 1} \, h^{q} 
				\|u\|_{\mcH^{q+1}(\Omega)}.
\end{equation}

For the analysis of linear solvers, like multigrid solvers~\cite{Takacs:2019a},
we also need low-order approximation error estimates. In the convergence proofs,
we usually have to estimate errors of the iterative scheme. Even if we know that the true
soultion satisfies certain regularity assumptions, this does not extend to
the errors. For them, we can only rely on the regularity statements
arizing from the domain, which means that $H^2$-regularity is usually the best
we can hope for. For this case, we obtain
\begin{equation}\nonumber
		\inf_{v_h \in V_h} \|u-v_h\|_{Q_h^+} \lesssim \sigma_0^{1/2}\,p^2\, h  
				\|u\|_{\mcH^2(\Omega)}.
\end{equation}
This means that
we obtain a quadratic increase in the spline degree $p$.
Using a refined analysis, we obtain as follows.
\begin{theorem}[Low-order approximation error estimate]\label{thrm:approx}
	Provided that $\sigma$ is
	as in Theorem~\ref{thrm:1}, then the estimate
	\begin{equation}\label{eq:thrm:approx}
		\inf_{v_h \in V_h} \|u-v_h\|_{Q_h^+} \lesssim  
		( \ln \sigma)^2  \, \sigma^{1/(2p_{\min}-1)} \, h  
				|u|_{\mcH^2(\Omega)}
	\end{equation}
   holds for all $u\in H^{1,\circ}(\Omega)\cap \mc H^2(\Omega)$.
\end{theorem}
The proof is given at the end of the next section.
Assuming again $p \eqsim p_{\min}\ge 2$, we obtain
\begin{equation}\label{eq:starstar}
	\|u-u_h\|_{Q_h} \lesssim 
	\inf_{v_h \in V_h} \|u-v_h\|_{Q_h^+} \lesssim 
		\sigma_0^{1/2}\;
		(\ln p)^2 \, h  
			|u|_{\mcH^2(\Omega)},
\end{equation}
i.e., an only poly-logarithmic increase in the spline degree $p$.

\section{Proof of the approximation error estimates}\label{sec:approx:proof}

Before we can give the proof, we give some auxiliary results. This section is
organized as follows. In Section~\ref{sec:4:1}, we give patch-wise projectors and estimates
for them. We introduce a mollifying operator and give estimates for that operator in
Section~\ref{sec:4:2}.
Finally, in Section~\ref{sec:4:3}, we give the proof for the approximation error estimate.

\subsection{Patch-wise projectors}\label{sec:4:1}

As first step, we recall the projection operators from~\cite[Sections~3.1 and 3.2]{Takacs:2018}.  
Let $\Pi_{p,Z}$ be the $H^1_D(0,1)$-orthogonal projection into $S_{p,Z}(0,1)$,
where
\[
	(u,v)_{H^1_D(0,1)} = (u',v')_{L_2(0,1)} + u(0)v(0).
\]
In what follows, we also write $S_{p,h}(0,1)$ and $\Pi_{p,h}$
if we refer to a uniform grid of size $h$.
\cite[Lemma~3.1]{Takacs:2018} states that $(\Pi_{p,Z}u)(0) = u(0)$ and
$(\Pi_{p,Z}u)(1) = u(1)$. Using $0=(u-\Pi_{p,Z}u,x^2)_{H^1_D(0,1)}=
2((u-\Pi_{p,Z}u)',x)_{L_2(0,1)}=-2(u-\Pi_{p,Z}u,1)_{L_2(0,1)} + u(1)-(\Pi_{p,Z}u)(1)$,
we obtain for $p\ge2$ and $u\in H^1(0,1)$ that
\begin{equation}\label{eq:15bb}
		(u-\Pi_{p,Z}u,1)_{L_2(0,1)} = 0. 
\end{equation}
The next step is to consider the multivariate case, more precisely the
parameter domain $\widehat{\Omega}=(0,1)^2$.
Let $\Pi^x_k:H^2(\widehat{\Omega})\rightarrow H^2(\widehat{\Omega})$ 
and $\Pi^y_k:H^2(\widehat{\Omega})\rightarrow H^2(\widehat{\Omega})$ be given by
\[
	(\Pi^x_ku)(x,y) = (\Pi_{p_k,Z_{k,1}}u(\cdot,y))(x)
	\quad\mbox{and}\quad
	(\Pi^y_ku)(x,y) = (\Pi_{p_k,Z_{k,2}}u(x,\cdot))(y)
\]
and let $\widehat{\Pi}_k:H^2(\widehat{\Omega}) \rightarrow S_{p_k,Z_{k,1}}(\widehat{\Omega})$ 
be such that
\begin{equation}\label{eq:3:8}
		\widehat{\Pi}_k = \Pi^x_k \Pi^y_k.
\end{equation}
For the physical domain, define $\Pi: H^{1,\circ}(\Omega)\cap \mcH^2(\Omega)\rightarrow V_h$ to be such that
\[
	(\Pi v)|_{\Omega_k} = (\widehat{\Pi}_k (v\circ G_k))\circ G_k^{-1}
	\quad\mbox{for all }
	v\in \mcH^2(\Omega)
	\mbox{ and } k=1,\ldots,K.
\]
Observe that we obtain using \eqref{eq:15bb} that
\begin{equation}\label{eq:15a}
		(u-\widehat{\Pi}_k u,1)_{L_2(\widehat{\Omega})} = 0, \qquad
		\widehat{\Pi}_k c = c \qquad \mbox{and}\qquad \Pi c = c .
\end{equation}
for all $c\in \mathbb{R}$.

The projectors $\widehat{\Pi}_k$ satisfy robust
error estimates and are almost stable in $H^2$.
\begin{lemma}\label{lem:univariate}
	Let $Z$ be a grid of size $h$, $p\in \{2,3,\ldots\}$, and
	$r\in \{1,2,\ldots,p\}$. Then,
	\begin{align*}
			\| (I - \Pi_{p,Z}) u \|_{L_2(0,1)}
				&\le  \pi^{-r-1} \; h^{r+1} |  u |_{H^{r+1}(0,1)}\\
			| (I - \Pi_{p,Z}) u |_{H^1(0,1)}
				&\le  \pi^{-r} \; h^r |  u |_{H^{r+1}(0,1)}
	\end{align*}
	hold for all $u\in H^{r+1}(0,1)$.
\end{lemma}
\begin{proof}
		The identity~\eqref{eq:15bb} implies that the projector $\Pi_{p,Z}$
		coincides with the projector $Q_p^1$
		from \cite[eq.~(3.8) and (3.9)]{Sande:Manni:Speleers:2018}.
		Thus, the desired result follows
		from~\cite[Theorem~3.1]{Sande:Manni:Speleers:2018}. 
\qed\end{proof}

\begin{lemma}\label{lem:h2:h2:1d}
	Let $Z$ be a quasi-uniform grid of size $h$,  $p\in \{2,3,\ldots\}$, and
	$r\in \{1,2,\ldots,p\}$. Then,
	\begin{align*}
			| (I - \Pi_{p,Z}) u |_{H^2(0,1)}
				&\lesssim p^2  \pi^{-r+1} \; h^{r-1} |  u |_{H^{r+1}(0,1)}
	\end{align*}
	holds for all $u\in H^{r+1}(0,1)$.
\end{lemma}
\begin{proof}
	The proof is analogous to the proof
	of~\cite[Theorem~4]{Hofreither:Takacs:2017}.
	Let $Q^2_p$ be the $H^2$-orthogonal projector into
	$S_{p,Z}(0,1)$ as introduced
	in~\cite{Sande:Manni:Speleers:2018}. Using
	the triangle inequality and a standard inverse estimate
	\cite[Corollary~3.94]{Schwab:1998}, we obtain
	\begin{align*}
			& | (I - \Pi_{p,Z}) u |_{H^2(0,1)}
			 \le
			| (I - Q^2_p) u |_{H^2(0,1)}
			+
			| (Q^2_p - \Pi_{p,Z}) u |_{H^2(0,1)} \\
			&  \quad \lesssim
			| (I - Q^2_p) u |_{H^2(0,1)}
			+
			h^{-1} p^2 | (Q^2_p - \Pi_{p,Z}) u |_{H^1(0,1)} \\
			& \quad \le
			| (I - Q^2_p) u |_{H^2(0,1)}
			+
			h^{-1} p^2 | (I - Q^2_p ) u |_{H^1(0,1)}
			+
			h^{-1} p^2 | (I - \Pi_{p,Z}) u |_{H^1(0,1)}.
	\end{align*}
	Thus, the desired result follows
		from~\cite[Theorem~3.1]{Sande:Manni:Speleers:2018}.
\qed\end{proof}

\begin{lemma}\label{lem:h1:h2}
	Let $r\in \{1,2,\ldots,p_{\min}\}$. The estimates
	\begin{align*}
			\| (I - \widehat{\Pi}_k) u \|_{L_2(\widehat\Omega)}
				&\lesssim  \pi^{-r-1} \; h_k^{r+1} |  u |_{H^{r+1}(\widehat\Omega)}\\
			| (I - \widehat{\Pi}_k) u |_{H^1(\widehat\Omega)}
				&\lesssim  \pi^{-r} \; h_k^r |  u |_{H^{r+1}(\widehat\Omega)}
\\
			| (I - \widehat{\Pi}_k) u |_{H^2(\widehat\Omega)}
				&\lesssim p^2 \pi^{-r+1} \; h_k^{r-1} |  u |_{H^{r+1}(\widehat\Omega)}
	\end{align*}
	hold for all $u\in H^{r+1}(\widehat\Omega)$.
\end{lemma}
\begin{proof}
		The proof is based on the univariate estimates given
		in the Lemmas~\ref{lem:univariate} and~\ref{lem:h2:h2:1d}
		($p_{\min}\ge 2$ and the quasi-uniformity of the grids
		have been required in~\eqref{eq:p2})
		and	follows the standard construction that can be found, e.g.,
		in the proof of \cite[Theorem~3.3]{Takacs:2018}.
\qed\end{proof}

On the interfaces, we have the following approximation error estimate.
\begin{lemma}\label{lem:h1:hr}
	Let $r\in \{1,2,\ldots,p_{\min}\}$ and
	$(k,l)\in \mcN\cup \mcN^*$. Then, the estimate
	$ \| (I - \widehat{\Pi}_k) u \|_{L_2(\widehat{I}_{k,l} )}
			\lesssim  \pi^{-r} h^r | u |_{H^r(\widehat{I}_{k,l} )}
	$ holds for all $u\in H^2(\widehat{\Omega}) \cap H^r(\widehat{I}_{k,l})$.
\end{lemma}
\begin{proof}
	Without loss of generality, we  assume $\widehat{I}_{k,l}=(0,1)\times \{0\}$.
	\cite[Theorem~3.4]{Takacs:2018} states
	that $((I - \widehat{\Pi}_k) u)(\cdot,0) = (I - \Pi_{p_k,Z_{k,1}}) (u(\cdot,0))$.
	So, we have
	\[
		\| (I - \widehat{\Pi}_k) u \|_{L_2({\widehat{I}_{k,l}})}^2
		=  \| (I - \Pi_{p_k,Z_{k,1}}) u(\cdot,0)\|_{L_2(0,1)}^2.
	\]
	Using~\cite[Eq.~(3.4)]{Takacs:2018}, \cite[Lemma~8]{Hofreither:Takacs:Zulehner:2017} and 
	that $\Pi_{p_k,Z_{k,1}}$ minimizes the $H^1$-seminorm,
	we further obtain
	\begin{align*}
		\| (I - \widehat{\Pi}_k) u \|_{L_2({\widehat{I}_{k,l}})}^2
			&\lesssim h^2 \inf_{v_h\in S_{p,h}(0,1)}
						| u(\cdot,0)-v_h |_{H^1(0,1)}^2.
	\end{align*}
	\cite[Theorem~3.1]{Sande:Manni:Speleers:2018} yields the
	desired result.
\qed\end{proof}

\subsection{A mollifying operator}\label{sec:4:2}

A second step of the proof is the introduction of a particular mollification
operator for the interfaces.

For $(k,l)\in \mcN$, let $\Upsilon_{k,l}$ be given
by $\Upsilon_{k,l} v := v \circ \gamma_{k,l}$. For $(k,l)\in \mcN^*$, we define
$\Upsilon_{k,l} v := v \circ G_k^{-1} \circ G_l \circ \gamma_{l,k}$, i.e., we have
$(\Upsilon_{k,l} v)(t) = v(\gamma_{l,k}(t))$ or 
$(\Upsilon_{k,l} v)(t) = v(\gamma_{l,k}(1-t))$, cf. Assumption~\ref{ass:geoequiv:bdy}.
For all cases,
$\Upsilon_{k,l}$ is a bijective function $H^s(0,1) \rightarrow H^s(\widehat{I}_{k,l})$
and  
\begin{equation}\label{eq:geoequiv:bdy}
	|  u|_{H^s(0,1)} \eqsim |\Upsilon_{k,l} u|_{H^s(\widehat{I}_{k,l})}
\end{equation}
holds for all $s$.
For $v\in H^s(\widehat{\Omega})$, we define the
abbreviated notation $\Upsilon_{k,l}^{-1}v
:=\Upsilon_{k,l}^{-1}(v|_{\widehat{I}_{k,l}})$ and observe 
\[
		\Upsilon_{k,l}^{-1}u \in H^{3/2}(0,1)
		\quad \mbox{for all}\quad
		u\in H^2(\widehat{\Omega}).
\]

For $(k,l)\in \mcN\cup\mcN^*$, we define \emph{extension operators}
$\Xi_{k,l}: H^s(\widehat{I}_{k,l}) \rightarrow H^s(\widehat{\Omega})$ by
\begin{align*}
	(\Xi_{k,l}w)(x,y):=
		\left\{
			\begin{array}{ll}
				\phi(x) w(0,y) & \mbox{ if } \widehat{I}_{k,l} = \{0\} \times (0,1) \\
				\phi(1-x) w(1,y) & \mbox{ if } \widehat{I}_{k,l} = \{1\} \times (0,1) \\
				\phi(y) w(x,0) & \mbox{ if } \widehat{I}_{k,l} = (0,1) \times \{0\} \\
				\phi(1-y) w(x,1) & \mbox{ if } \widehat{I}_{k,l} = (0,1) \times \{1\} \\
			\end{array}
		\right.
	,
\end{align*}
where
\begin{align}
		\phi(x)  := \max\{0,1-\eta^{-1} x\} \quad \mbox{and} \quad 
		\eta \in (0,1)
		\label{eq:def:phi}. 
\end{align}

Now, define for each patch $\Omega_k$ a mollifying operator $\widehat{\mc{M}}_k$ by 
\begin{equation}\label{eq:m:hat:def}
		\widehat{\mc{M}}_k  := I - \sum_{l\in \mcN(k)} \Xi_{k,l}
		\Upsilon_{k,l}
				(I-\Pi_{r,\eta})
		\Upsilon_{k,l}^{-1}  .
\end{equation}
The combination of the patch local operators yields a global 
operator $\mc{M}$:
\begin{equation}\label{eq:glob:mcM:def}
	(\mc{M} u)|_{\Omega_k} := (\widehat{\mc{M}}_k(u\circ G_k))\circ G_k^{-1}.
\end{equation}
Observe that $\mc{M}$ preserves constants, i.e.,
\begin{equation}\label{eq:mcM:c}
		\mc{M} c = c
		\quad \mbox{for all}\quad
		c \in \mathbb R.
\end{equation}

\begin{lemma}\label{lem:mathcalM:0a}
		For all $(k,l) \in \mcN\cup\mcN^*$ and all $u\in H^1_0(\widehat{I}_{k,l}):=\{u\in H^1
		(\widehat{I}_{k,l}): u = 0 \mbox{ on } \partial \widehat{I}_{k,l} \}$, we have
		\[
			\Xi_{k,l} u = 0 \quad \mbox{on} \quad \partial \widehat{\Omega} \,\backslash\,
													\widehat{I}_{k,l}
		\quad
		\mbox{and}
		\quad
			\Xi_{k,l} u = u \quad \mbox{on} \quad \widehat{I}_{k,l}.
		\]
\end{lemma}
\begin{proof}
	Assume without loss of generality that $\widehat{I}_{k,l}=\{0\}\times (0,1)$.
	For this case, we have
	\[
			(\Xi_{k,l} u)(x,y) = \phi(x)u(0,y).
	\]
	As $u\in H^1_0(\widehat{I}_{k,l})$, we obtain $u(0,0)=u(0,1)=0$. This
	shows the first statement for the two boundary segments adjacent to
	$\widehat{I}_{k,l}$, i.e., $[0,1]\times\{0\}$ and $[0,1]\times\{1\}$. 
	Since $\eta <1$ yields $\phi(1)=0$, we also have the first statement for the boundary
	segment $\{1\}\times (0,1)$. This finishes the proof for the first statement.
	The proof for the second statement follows directly from $\phi(0)=1$.
\qed\end{proof}

\begin{lemma}\label{lem:mathcalM:0b}
		$
			\Upsilon_{k,l}^{-1}  \widehat{\mc{M}}_k = 
				\Pi_{r,\eta} \Upsilon_{k,l}^{-1}  
		$
		holds
		for all $(k,l) \in \mcN\cup\mcN^*$.
\end{lemma}
\begin{proof}
	\eqref{eq:m:hat:def} implies
	$\Upsilon_{k,l}^{-1}  \widehat{\mc{M}}_k =
		\Upsilon_{k,l}^{-1}  - \sum_{j\in \mcN(k)} 
		\Upsilon_{k,l}^{-1}  \Xi_{k,j}
		\Upsilon_{k,j}
				(I-\Pi_{r,\eta})
		\Upsilon_{k,j}^{-1}$. 
		Observe that the projector $\Pi_{r,\eta}$ is interpolatory on the
		boundary (\cite[Lemma~3.1]{Takacs:2018}). So, $
				(I-\Pi_{r,\eta})$ maps into $H^1_0(0,1)$
				and $\Upsilon_{k,j}
				(I-\Pi_{r,\eta})$ maps into $H^1_0(\widehat{I}_{k,j})$.
				Therefore,
		Lemma~\ref{lem:mathcalM:0a} yields
		$\Upsilon_{k,l}^{-1}  \widehat{M}_k =
		\Upsilon_{k,l}^{-1}  - 
		\Upsilon_{k,l}^{-1}  
		\Upsilon_{k,l}
				(I-\Pi_{r,\eta})
		\Upsilon_{k,l}^{-1} $, which immediately implies the desired result.
\qed\end{proof}

Before we proceed, we give a certain trace like estimate.
\begin{lemma}\label{lem:interpol}
	The estimate
	\[
			\Psi(u):=\inf_{v \in H^1(\widehat{I}_{k,l})} \|u-v\|_{L_2(\widehat{I}_{k,l})}^2
				+ \theta^2 |v|_{H^1(\widehat{I}_{k,l})}^2 
				\lesssim \theta |u|_{H^1(\widehat{\Omega})}^2
	\]
	holds for all $u\in H^1(\widehat{\Omega})$ and $(k,l)\in \mcN \cup \mcN^*$
	and all $\theta > 0$.
\end{lemma}
\begin{proof}
	A trace theorem~\cite[Lemma~4.4]{Takacs:2018} yields
	\begin{equation}\label{eq:lem:interpol}
		\Psi(u)
			\lesssim \inf_{v \in H^2(\widehat{\Omega})} 
							\|u-v\|_{L_2(\widehat{\Omega})} |u-v |_{H^1(\widehat{\Omega})}
				+ \theta^2 |v|_{H^1(\widehat{\Omega})}|v|_{H^2(\widehat{\Omega})}.
	\end{equation}
	\emph{Case 1.} Assume $\theta < 1$. In this case, we choose
	$v$ to be the $H^1$-orthogonal projection of $u$ into $S_{3,\lceil \theta^{-1} \rceil^{-1}}
	(\widehat{\Omega})$. Since the spline degree of that space is fixed,
	we obtain using a standard inverse inequality (\cite[Corollary~3.94]{Schwab:1998})
	and a standard approximation error estimate (like from~\cite{Takacs:Takacs:2015})
	that
	\[
			\Psi(u) \lesssim (\lceil \theta^{-1} \rceil^{-1}+\theta^2 \lceil \theta^{-1} \rceil) |v|_{H^1(\widehat{\Omega})}^2 \lesssim \theta |v|_{H^1(\widehat{\Omega})}^2.
	\]
	\emph{Case 2.} Assume $\theta \ge 1$. In this case, we choose
	$v := (u,1)_{L_2(\Omega)}$ and obtain from~\eqref{eq:lem:interpol} directly
	\[
		\Psi(u)
			\lesssim  
							\|u-v\|_{H^1(\widehat{\Omega})} |u|_{H^1(\widehat{\Omega})}.
	\]
	In this case, the Poincar\'e inequality finishes the proof.
\qed\end{proof}

As a next step, we show that the mollifier constructs functions that are
very smooth on the interfaces.
\begin{lemma}\label{lem:mathcalM:2} 
	The estimate 
	$
		| \widehat{\mc{M}}_k \widehat{u}_k |_{H^r(\widehat{I}_{k,l})}^2
				\lesssim (2 \sqrt{3} r^2\eta^{-1} )^{2r-3} r^2
						|  \widehat{u}_k |_{H^2(\widehat{\Omega})}^2
	$
	holds for all $\widehat{u}_k\in H^r(\widehat{\Omega})$ and all
	$(k,l)\in \mcN \cup \mcN^*$. 
\end{lemma}
\begin{proof}
	We have using~\eqref{eq:geoequiv:bdy}
	and Lemma~\ref{lem:mathcalM:0b}
	\[
		| \widehat{\mc{M}}_k \widehat{u}_k |_{H^r(\widehat{I}_{k,l})}^2
			\eqsim
			| \Upsilon_{k,l}^{-1}  \widehat{\mc{M}}_k \widehat{u}_k |_{H^r(0,1)}^2
			= 
			| \Pi_{r,\eta} \Upsilon_{k,l}^{-1}  \widehat{u}_k |_{H^r(0,1)}^2.
	\]
	Now, a standard inverse estimate (\cite[Corollary~3.94]{Schwab:1998}) 
	yields
	\[
		| \widehat{\mc{M}}_k \widehat{u}_k |_{H^r(\widehat{I}_{k,l})}^2
			\lesssim
			\psi^{2(r-s)} | \Pi_{r,\eta} \Upsilon_{k,l}^{-1}  \widehat{u}_k |_{H^{s}(0,1)}^2 \quad \mbox{for}\quad s\in\{1,2\},
	\]
	where $\psi:= 2\sqrt{3} r^2 \eta^{-1} $.
	Lemma~\ref{lem:h2:h2:1d} and the $H^1$-stability of $\Pi_{r,\eta}$ yield
	\[
			| \Pi_{r,\eta} w|_{H^2(0,1)}^2 \lesssim r^4 |w|_{H^2(0,1)}^2
			\quad \mbox{and} \quad
			| \Pi_{r,\eta} w|_{H^1(0,1)}^2 \le |w|_{H^1(0,1)}^2,
	\]
	so we obtain 
	\begin{align*}
		| \widehat{\mc{M}}_k \widehat{u}_k |_{H^r(\widehat{I}_{k,l})}^2
			\lesssim
			\psi^{2r-2} 
			\quad \inf_{v \in H^2(0,1)}
			 \big( |  \Upsilon_{k,l}^{-1}  \widehat{u}_k - v |_{H^{1}(0,1)}^2+
			\psi^{-2} r^{4}
				|  v |_{H^{2}(0,1)}^2 \big).
	\end{align*}
	Using~\eqref{eq:geoequiv:bdy}, we obtain
	\begin{align*}
		| \widehat{\mc{M}}_k \widehat{u}_k |_{H^r(\widehat{I}_{k,l})}^2
			\lesssim
			\psi^{2r-2}
			\inf_{v \in H^2(\widehat{I}_{k,l})}
			 \big(
				|   \widehat{u}_k - v |_{H^{1}(\widehat{I}_{k,l})}^2+
			\psi^{-2} r^{4}
				|  v |_{H^{2}(\widehat{I}_{k,l})}^2 \big).
	\end{align*}
	By applying Lemma~\ref{lem:interpol} to the derivative of $\widehat{u}_k$,
	we obtain the desired result.
\qed\end{proof}

\begin{lemma}\label{lem:mathcalM:0}
		$
			\| \llbracket (I-\mc{M}) u  \rrbracket \|_{L_2(I_{k,l})} = 0
		$
		holds for all $u\in H^{1,\circ}(\Omega)\cap \mcH^2(\Omega)$
		and $(k,l) \in \mcN$.
\end{lemma}
\begin{proof}
	Let $u\in H^{1,\circ}(\Omega)\cap \mcH^2(\Omega)$ be arbitrary but fixed. 

	We obtain using the definition of $\Upsilon_{k,l}$ and
	$\Upsilon_{l,k}$ and Lemma~\ref{lem:geoequiv:1} that
	\begin{equation}\label{eq:lem:mathcalM:0}
	\begin{aligned}
		\|\llbracket w \rrbracket \|_{L_2(I_{k,l})}
		& = \| w_k-w_l \|_{L_2(I_{k,l})}
		\eqsim \| \widehat{w}_k-\widehat{w}_l\circ G_l^{-1}\circ G_k \|_{L_2(\widehat{I}_{k,l})} \qquad\\
		&= \| \Upsilon_{k,l}^{-1}(\widehat{w}_k-\widehat{w}_l\circ G_l^{-1}\circ G_k) \|_{L_2(0,1)} \\
		& = \| \Upsilon_{k,l}^{-1}\widehat{w}_k - \Upsilon_{l,k}^{-1} \widehat{w}_l \|_{L_2(0,1)} 
	\end{aligned}
	\end{equation}
	holds,
	where $\widehat{w}_k:=w_k\circ G_k$ and $\widehat{w}_l:=w_l\circ G_l$. Since
	$u\in H^1(\Omega)$, a standard trace theorem yields
	\begin{equation}\label{lem:mathcalM:0:proof:1}
			 \Upsilon_{k,l}^{-1} \widehat{u}_k 
						- \Upsilon_{l,k}^{-1} \widehat{u}_l = 0.
	\end{equation}
	Thus,~\eqref{eq:lem:mathcalM:0} implies 
	$\| \llbracket u  \rrbracket \|_{L_2(I_{k,l})} = 0$.
	By plugging $\mc{M}u$ into~\eqref{eq:lem:mathcalM:0}, we obtain using Lemma~\ref{lem:mathcalM:0b}
	\[
		\|\llbracket \mc{M} u \rrbracket \|_{L_2(I_{k,l})}
			\eqsim \| \Pi_{r,\eta} ( \Upsilon_{k,l}^{-1} \widehat{u}_k 
						- \Upsilon_{l,k}^{-1}\widehat{u}_l )   \|_{L_2(0,1)}.
	\]
	Using~\eqref{lem:mathcalM:0:proof:1} and $\Pi_{r,\eta}0=0$, we obtain
	$\|\llbracket \mc{M} u \rrbracket \|_{L_2(I_{k,l})}=0$ and consequently also
	$\|\llbracket (I-\mc{M}) u \rrbracket \|_{L_2(I_{k,l})}=0$.
\qed\end{proof}

\begin{lemma}\label{lem:mathcalM:3} The estimate
	$
			 \| \widehat{\Pi}_k (I-\widehat{\mc{M}}_k) u \|_{H^{1,\circ}(\widehat{\Omega})}^2
				\lesssim (1+\eta^2 h^{-2})h^2 
						| u |_{H^{2}(\widehat{\Omega})}^2
	$
	holds for all $u\in H^2(\widehat{\Omega})$ and $k=1,\ldots,K$.
\end{lemma}
\begin{proof}
	Using the definition of $\widehat{\mc{M}}_k$ and of the $H^{1,\circ}$-norm, we obtain 
	\begin{align}\nonumber
		 \| \widehat{\Pi}_k (I-\widehat{\mc{M}}_k) u \|_{H^{1,\circ}(\widehat{\Omega})}
		 	  &\le \sum_{l\in \mcN(k)} 
		 	\| \widehat{\Pi}_k \Xi_{k,l} \Upsilon_{k,l}
		 				(I-\Pi_{r,\eta}) \Upsilon_{k,l}^{-1}  u \|_{H^{1,\circ}(\widehat{\Omega})} \\
		 	 & \lesssim \sum_{l\in \mcN(k)} (\Psi_{x,l} + \Psi_{y,l} + \Psi_{\circ,l}),\label{eq:molliproof1}
	\end{align}
	where
	$
		 	\Psi_{\square,l}
		 	 :=
		 		\| \tfrac{\partial}{\partial \square} \widehat{\Pi}_k \Xi_{k,l} \Upsilon_{k,l}
		 				(I-\Pi_{r,\eta}) \Upsilon_{k,l}^{-1}  u\|_{L_2(\widehat{\Omega})}
	$
	for $\square\in\{x,y\}$ and
	$
		 	\Psi_{\circ,l}
		 	 :=
		 		(  \widehat{\Pi}_k \Xi_{k,l} \Upsilon_{k,l}
		 				(I-\Pi_{r,\eta}) \Upsilon_{k,l}^{-1}  u,1)_{L_2(\widehat{\Omega})}.
	$
	We estimate the terms $\Psi_{x,l}$, $\Psi_{y,l}$ and $\Psi_{\circ,l}$ separately. 
	Let without loss of generality $\widehat{I}_{k,l}= \{0\}\times (0,1)$. 

	\emph{Step~1.} Using~\eqref{eq:3:8} and the $H^1$-stability of the $H^{1,D}$-orthogonal projection,
	and $w:=\Upsilon_{k,l}
		 				(I-\Pi_{r,\eta})\Upsilon_{k,l}^{-1}  u$,  we obtain
	\begin{align*}
		\Psi_{x,l}^2 &=
		\| \tfrac{\partial}{\partial x} \Pi^x_k\Pi^y_k \Xi_{k,l}  w\|_{L_2(\widehat{\Omega})}^2 
		\le 
		\| \tfrac{\partial}{\partial x} \Pi^y_k \Xi_{k,l} w\|_{L_2(\widehat{\Omega})}^2\\
		&= \int_0^1 \int_0^1 \big(\phi'(x) \; (\Pi_{p_k,Z_{k,2}}w(0,\cdot))(y) \big)^2 \, \dd x\, \dd y \\
		& =| \phi |_{H^1(0,1)}^2 \| \Pi_{p_k,Z_{k,2}}(w(0,\cdot))\|_{L_2(0,1)}^2
		\eqsim \hat{\Psi}_{x,l}^2 :=  \eta^{-1}\| \Pi_{p_k,Z_{k,2}}(w(0,\cdot))\|_{L_2(0,1)}^2, 
	\end{align*}
	where we use $|\phi|_{H^1(0,1)}^2 \eqsim \eta^{-1}$.   
	The triangle inequality yields
	\begin{align*}
		\Psi_{x,l}^2 \le \hat{\Psi}_{x,l}^2  
				\lesssim \eta^{-1}\|w(0,\cdot)\|_{L_2(0,1)}^2
					+\eta^{-1}\| (I-\Pi_{p_k,Z_{k,2}})(w(0,\cdot))\|_{L_2(0,1)}^2.
	\end{align*}
	Lemma~\ref{lem:univariate} yields
	\begin{align*}
		\Psi_{x,l}^2 \le \hat{\Psi}_{x,l}^2  
				\lesssim \eta^{-1}\|w(0,\cdot)\|_{L_2(0,1)}^2
					+\eta^{-1}h^2 | w(0,\cdot)|_{H^1(0,1)}^2.
	\end{align*}
	The definition of $w$ and~\eqref{eq:geoequiv:bdy} yield
	\begin{align*}
		\Psi_{x,l}^2 \le \hat{\Psi}_{x,l}^2
				\lesssim \eta^{-1}\|(I-\Pi_{r,\eta})\Upsilon_{k,l}^{-1}  u\|_{L_2(0,1)}^2
					+\eta^{-1}h^2 | (I-\Pi_{r,\eta})\Upsilon_{k,l}^{-1}  u|_{H^1(0,1)}^2.
	\end{align*}
	Lemma~\ref{lem:univariate} yields
	\begin{align*}
		\Psi_{x,l}^2 \le \hat{\Psi}_{x,l}^2  
				\lesssim (\eta+\eta^{-1}h^2) 
				\Big(
				\inf_{v\in H^2(0,1)}
				 | \Upsilon_{k,l}^{-1} u-v|_{H^1(0,1)}^2 +
				\eta^{2} | v|_{H^2(0,1)}^2 \Big).
	\end{align*}
	The equation~\eqref{eq:geoequiv:bdy} yields further
	\begin{align*}
		\Psi_{x,l}^2 \le \hat{\Psi}_{x,l}^2
				\lesssim (\eta+\eta^{-1}h^2) \Big(
				\inf_{v\in H^2(\widehat{I}_{k,l})}
				 | u-v|_{H^1(\widehat{I}_{k,l})}^2 +
				\eta^{2} | v|_{H^2(\widehat{I}_{k,l})}^2 \Big) .
	\end{align*}
	Now, Lemma~\ref{lem:interpol} applied to the derivative of $u$ yields
	\begin{align}\label{eq:molliproof2}
		\Psi_{x,l}^2 \le \hat{\Psi}_{x,l}^2
				&\lesssim  (\eta+\eta^{-1}h^2) \eta | u|_{H^{2}(\widehat{\Omega})}^2 
				= (1+\eta^{2}h^{-2})h^2 | u|_{H^{2}(\widehat{\Omega})}^2.
	\end{align}
	
	\emph{Step~2.} Using~\eqref{eq:3:8} and the $H^1$-stability of the $H^{1,D}$-orthogonal
	projection and $w:=\Upsilon_{k,l}
		 				(I-\Pi_{r,\eta})\Upsilon_{k,l}^{-1}  u$,  we obtain
	\begin{align*}
		\Psi_{y,l}^2 &=
		\|\tfrac{\partial}{\partial y} \Pi^x_k\Pi^y_k \Xi_{k,l} w \|_{L_2(\widehat{\Omega})}^2 
		 	 \le 
		 	\|\tfrac{\partial}{\partial y} \Pi^x_k \Xi_{k,l} w \|_{L_2(\widehat{\Omega})}^2\\
			&= \int_0^1 \int_0^1 \big((\Pi_{p_k,Z_{k,1}}\phi)(x) \;\tfrac{\partial}{\partial y} w(0,y) \big)^2 \, \dd x\, \dd y 
		 	  \eqsim 
		 	 \| \Pi_{p_k,Z_{k,1}}  \phi \|_{L_2(0,1)}^2
		 	 | w |_{H^1(\widehat{I}_{k,l})}^2.
	\end{align*}
	Using $\|\Pi_{p_k,Z_{k,1}} \phi\|_{L_2(0,1)}^2 \lesssim \| \phi\|_{L_2(0,1)}^2+\|(I-\Pi_{p_k,Z_{k,1}}) \phi\|_{L_2(0,1)}^2
	\lesssim \| \phi\|_{L_2(0,1)}^2+h^2| \phi|_{H^1(0,1)}^2 \eqsim 
	\eta + h^2 \eta^{-1} =
	(\eta^2 h^{-2}+1)h^2\eta^{-1}$ and the definition of $w$, we obtain
	\[
			\Psi_{y,l}^2\lesssim (\eta^2 h^{-2}+1)h^2\eta^{-1} | \Upsilon_{k,l}
		 				(I-\Pi_{r,\eta})\Upsilon_{k,l}^{-1}  u |_{H^1(\widehat{I}_{k,l})}^2.
	\]
	Using~\eqref{eq:geoequiv:bdy}, we obtain further
	\[
			\Psi_{y,l}^2\lesssim (\eta^2 h^{-2}+1)h^2\eta^{-1} | 
		 				(I-\Pi_{r,\eta})\Upsilon_{k,l}^{-1}  u |_{H^1(0,1)}^2.
	\]
	Using the $H^1$-stability of $\Pi_{r,\eta}$ and the
	approximation error estimate~\cite[Theorem~3.1]{Takacs:2018}, we obtain
	\[
			\Psi_{y,l}^2\lesssim (\eta^2 h^{-2}+1)h^2
						\big(
						\inf_{v\in H^2(0,1)}
						\eta^{-1} 
		 				|\Upsilon_{k,l}^{-1}  u-v |_{H^1(0,1)}^2
						+ \eta 
		 				|v |_{H^2(0,1)}^2
		 				\big).
	\]
	Using~\eqref{eq:geoequiv:bdy} and Lemma~\ref{lem:interpol} applied to the derivative
	of $\Upsilon_{k,l}^{-1}  u$, we obtain
	\begin{align}\nonumber
	\Psi_{y,l}^2
			&\lesssim (\eta^2 h^{-2}+1)h^2
						\big(
						\inf_{v\in H^2(\widehat{I}_{k,l})}
						\eta^{-1} 
		 				| u-v |_{H^1(\widehat{I}_{k,l})}^2
						+ \eta 
		 				|v |_{H^2(\widehat{I}_{k,l})}^2
		 				\big) \\
		 	&\lesssim
		 	(\eta^2h^{-2}+1)h^2 | u |_{H^{2}(\widehat{\Omega})}^2.
		 	\label{eq:molliproof3}
	\end{align}

	\emph{Step~3.} Using 
	$w:=\Upsilon_{k,l}	(I-\Pi_{r,\eta})\Upsilon_{k,l}^{-1} u$
	and~\eqref{eq:15bb}, we obtain
	\begin{align*}
		&\Psi_{\circ,l}^2 =
		(  \Pi^x_k\Pi^y_k \Xi_{k,l}  w,1)_{L_2(\widehat{\Omega})}^2 
		 = \int_0^1\int_0^1
			(\Pi_{p_k,Z_{k,1}}\phi)(x) (\Pi_{p_k,Z_{k,2}} w(0,\cdot))(y) 
		\dd x \dd y \\
		& = \int_0^1
			(\Pi_{p_k,Z_{k,1}}\phi)(x) \dd x
			\int_0^1(\Pi_{p_k,Z_{k,2}} w(0,\cdot))(y) 
		 \dd y \\
        &= \int_0^1
			\phi(x) \dd x
			\int_0^1(\Pi_{p_k,Z_{k,2}} w(0,\cdot))(y) 
		 \dd y 
		  \le \frac{\eta}{2}\;
			\|\Pi_{p_k,Z_{k,2}} w(0,\cdot)\|^2_{L_2(0,1)} = \frac{\eta^2}{2} \hat{\Psi}_{x,l}^2.
	\end{align*}
	Using~\eqref{eq:molliproof2}, and using $\eta\le1$, we further obtain
	\begin{align}\label{eq:molliproof4}
		\Psi_{\circ,l}^2 \lesssim 
			(1+\eta^2 h^{-2} )  h^2 |u|_{H^2(\widehat{\Omega})}^2.
	\end{align}
	
	\emph{Concluding step.} The combination of~\eqref{eq:molliproof1},
	\eqref{eq:molliproof2}, \eqref{eq:molliproof3},
	and~\eqref{eq:molliproof4} yields the desired result.
\qed\end{proof}

\subsection{The approximation error estimate}\label{sec:4:3}

The following three lemmas give approximation error estimates~\eqref{eq:thrm:approx}
for the choice $u_h:=\Pi\mc{M}u$ separately for the individual parts of
$\|\cdot\|_{Q_h^+}$.

\begin{lemma}\label{lem:parts:1}
	$| (I-\Pi\mc{M}) u |_{\mcH^1(\Omega)}^2 \le (1+\eta^2 h^{-2}) h^2 |u|_{\mcH^2(\Omega)}^2$ holds for all $u\in H^{1,\circ}(\Omega)\cap \mcH^2(\Omega)$.
\end{lemma}
\begin{proof}
	First note that the Poincar\'e inequality yields
	\begin{equation}\label{eq:poinc}
		\|\cdot\|_{H^1(\widehat{\Omega})}^2
			\eqsim \|\cdot\|_{H^{1,\circ}(\widehat{\Omega})}^2 := |\cdot|_{H^1(\widehat{\Omega})}^2
				+ (\cdot,1)_{L_2(\widehat{\Omega})}^2.
	\end{equation}
	Let $u\in H^{1,\circ}(\Omega)\cap \mcH^2(\Omega)$ be arbitrary but fixed and let $\widehat{u}_k := u \circ G_k$.
	Using Lemma~\ref{lem:geoequiv:1}, the triangle inequality,
	\eqref{eq:poinc} and~\eqref{eq:15a}, we obtain 
	\begin{align*}
		& | (I-\Pi\mc{M}) u |_{\mcH^1(\Omega)}^2
			 \lesssim 
				\sum_{k=1}^K  | (I-\widehat{\Pi}_k) \widehat{u}_k |_{H^1(\widehat\Omega)}^2
				+ \sum_{k=1}^K   \|  \widehat{\Pi}_k(I- \widehat{\mc{M}}_k) \widehat{u}_k \|_{H^{1,\circ}(\widehat\Omega)}^2 .
	\end{align*}
	We further obtain using Lemma~\ref{lem:h1:h2} and Lemma~\ref{lem:mathcalM:3},
	\begin{align*}
	| (I-\Pi\mc{M}) u |_{\mcH^1(\Omega)}^2
			& \lesssim 
				h^2 \sum_{k=1}^K  |  \widehat{u}_k |_{H^2(\widehat\Omega)}^2
				+ \sum_{k=1}^K (1+\eta^2 h^{-2}) h^2   \|\widehat{u}_k \|_{H^2(\widehat\Omega)}^2 	.
	\end{align*}
	Lemma~\ref{lem:geoequiv:1}, \eqref{eq:15a}, \eqref{eq:mcM:c} and the Poincar\'e inequality
	finish the proof.
\qed\end{proof}

\begin{lemma}\label{lem:parts:2}
	$| (I-\Pi\mc{M}) u |_{\mcH^2(\Omega)}^2 \le (1+\eta^2 h^{-2}) p^4 | u |_{\mcH^2(\Omega)}^2$ holds for all $u\in H^{1,\circ}(\Omega)\cap \mcH^2(\Omega)$.
\end{lemma}
\begin{proof}
	Using Lemma~\ref{lem:geoequiv:1},
	the triangle inequality, \eqref{eq:poinc}, \eqref{eq:15a}
	and a standard inverse inequality (\cite[Corollary~3.94]{Schwab:1998}), we
	obtain
	\begin{align*}
		 &| (I-\Pi\mc{M}) u |_{\mcH^2(\Omega)}^2
		 		 \lesssim \sum_{k=1}^K  
		   		\|(I-\widehat{\Pi}_k\widehat{\mc{M}}_k ) \widehat{u}_k \|_{H^2(\widehat{\Omega})}^2
		 		\\
		 		&\lesssim \sum_{k=1}^K 
		   		\|(I-\widehat{\Pi}_k ) \widehat{u}_k \|_{H^2(\widehat{\Omega})}^2
		   		+\sum_{k=1}^K 
		 		\|\widehat{\Pi}_k (I-\widehat{\mc{M}}_k ) \widehat{u}_k \|_{H^2(\widehat{\Omega})}^2
		 		\\
		 		&\lesssim \sum_{k=1}^K 
		   		|(I-\widehat{\Pi}_k ) \widehat{u}_k |_{H^2(\widehat{\Omega})}^2
		   		+p^4h^{-2} \sum_{k=1}^K 
		 		\|\widehat{\Pi}_k (I-\widehat{\mc{M}}_k ) \widehat{u}_k \|_{H^{1,\circ}(\widehat{\Omega})}^2.
	\end{align*}
	By again applying 
    Lemma~\ref{lem:h1:h2} and Lemma~\ref{lem:mathcalM:3}, we obtain
	\begin{align*}
		 | (I-\Pi\mc{M}) u |_{\mcH^2(\Omega)}^2
		 		&\lesssim p^2 \sum_{k=1}^K  
		   		| \widehat{u}_k |_{H^2(\widehat{\Omega})}^2
		   		+p^4 (1+\eta^2 h^{-2}) \sum_{k=1}^K 
		 		\| \widehat{u}_k \|_{H^{2}(\widehat{\Omega})}^2\\
		 		& \lesssim p^4 (1+\eta^2 h^{-2}) \sum_{k=1}^K 
		 		\| \widehat{u}_k \|_{H^{2}(\widehat{\Omega})}^2.
	\end{align*}
	Lemma~\ref{lem:geoequiv:1}, \eqref{eq:15a}, \eqref{eq:mcM:c} and the Poincar\'e inequality
	finish the proof.
\qed\end{proof}

\begin{lemma}\label{lem:parts:3}
	$\sum_{(k,l)\in \mc{I}}  \| \llbracket (I-\Pi\mc{M}) u \rrbracket \|_{L_2(I_{k,l})}^2
			\le \pi^{-2r} h^{2r} (2 \sqrt{3} r^2 \eta^{-1})^{2r-3}r^2 |u|_{\mcH^2(\Omega)}^2$
	holds for all $u\in H^{1,\circ}(\Omega)\cap \mcH^2(\Omega)$.
\end{lemma}
\begin{proof}
	Let $u\in H^{1,\circ}(\Omega)\cap \mcH^2(\Omega)$ be arbitrary but fixed and let $\widehat{u}_k := u \circ G_k$.
	Observe that the triangle inequality, Lemma~\ref{lem:mathcalM:0} 
	and Lemma~\ref{lem:geoequiv:1} yield
	\begin{align*}
		&\sum_{(k,l)\in \mcN} \| \llbracket (I-\Pi\mc{M}) u \rrbracket \|_{L_2(I_{k,l})}^2
		   \lesssim 
				 \sum_{(k,l) \in \mcN} \| \llbracket (I-\Pi) \mc{M} u \rrbracket \|_{L_2(I_{k,l})}^2 \\
		& \quad \lesssim 
				 \sum_{(k,l)\in \mcN\cup \mcN^*} \|  ((I-\Pi) \mc{M} u)|_{\Omega_k}  \|_{L_2(I_{k,l})}^2 \\
		& \quad \lesssim 
				 \sum_{(k,l)\in \mcN\cup \mcN^*} \|  (I-\widehat{\Pi}_k) \widehat{\mc{M}}_k \widehat{u}_k  \|_{L_2(\widehat{I}_{k,l})}^2.
	\end{align*}
	Lemma~\ref{lem:h1:hr} yields 
	\begin{align*}
		\sum_{(k,l)\in \mc{I}}\| \llbracket (I-\Pi\mc{M}) u \rrbracket \|_{L_2(I_{k,l})}^2
		  \lesssim \pi^{-2r} h^{2r}
				 \sum_{(k,l)\in \mcN\cup \mcN^*} 
				 		|   \widehat{\mc{M}}_k \widehat{u}_k  |_{H^r(\widehat{I}_{k,l})}^2.
	\end{align*}
	Since Assumption~\ref{ass:geoequiv:bdy} yields
	$|   \widehat{\mc{M}}_k \widehat{u}_k  |_{H^r(\widehat{I}_{k,l})}^2 \eqsim
	|   \widehat{\mc{M}}_l \widehat{u}_l  |_{H^r(\widehat{I}_{l,k})}^2$, we obtain
	\[
				 		|   \widehat{\mc{M}}_k \widehat{u}_k  |_{H^r(\widehat{I}_{k,l})}^2
			+ 
				 		|   \widehat{\mc{M}}_l \widehat{u}_l  |_{H^r(\widehat{I}_{l,k})}^2
			\lesssim 
				 		|   \widehat{\mc{M}}_k \widehat{u}_k  |_{H^r(\widehat{I}_{k,l})}^2
			+ 
				 		|   \widehat{\mc{M}}_l \widehat{u}_l  |_{H^r(\widehat{I}_{l,k})}^2
	\]
	and therefore also
	\begin{align*}
		\sum_{(k,l)\in \mc{I}} \| \llbracket (I-\Pi\mc{M}) u \rrbracket \|_{L_2(I_{k,l})}^2
		  \lesssim \pi^{-2r} h^{2r}
				 \sum_{(k,l)\in \mcN\cup \mcN^*} 
				 		|   \widehat{\mc{M}}_k \widehat{u}_k  |_{H^r(\widehat{I}_{k,l})}^2.
	\end{align*}	
	Now, Lemma~\ref{lem:mathcalM:2} yields
	\begin{align*}
		\sum_{(k,l)\in \mc{I}}  \| \llbracket (I-\Pi\mc{M}) u \rrbracket \|_{L_2(I_{k,l})}^2
		  \lesssim \pi^{-2r} h^{2r} (2 \sqrt{3} r^2 \eta^{-1})^{2r-3}r^2
				 \sum_{k=1}^K 
				 		| \widehat{u}_k |_{H^2(\widehat{\Omega})}^2.
	\end{align*}
	Lemma~\ref{lem:geoequiv:1} finishes the proof.
\qed\end{proof}

Finally, we can show Theorems~\ref{thrm:high} and~\ref{thrm:approx}. 
\begin{proof}[of Theorem~\ref{thrm:approx}]
	Let $u\in \mcH^{2,\circ}(\Omega) \cap H^1(\Omega)$ be arbitrary but
	fixed and define $u_h:=\Pi\mc{M}u$.
	
	First, we show that
	\begin{equation}\label{eq:thrm:approx:1}
		\| u-u_h \|_{Q_h^+}^2 \lesssim \underbrace{ \big(1 + \eta_0^2 h^{-2} 
					+ \sigma (3 r^2 h\eta_0^{-1})^{2r-3} r^2 \big)}_{\displaystyle \Psi:=  } h^2
						|u|_{\mcH^2(\Omega)}^2
	\end{equation}
	holds for any $r \in \{2,\ldots, p_{\min}\}$ and all $\eta_0>0$.
	
	\emph{Case 1.} Assume $\eta_0 \le 1$. In this, case we
	define $\eta := \lceil \eta_0^{-1} \rceil^{-1}$ and observe
	\[
					\frac12 \eta_0 \le \eta \le \eta_0.
	\]
	\eqref{eq:def:qhplus} and Lemmas~\ref{lem:parts:1}, \ref{lem:parts:2}, and \ref{lem:parts:3}
	and $\sigma \gtrsim p^2$ yield
	\begin{align*}
		& \| u-u_h \|_{Q_h^+}^2 
			 \lesssim 
				\left(1+\eta^2 h^{-2}+\frac{\sigma}{h^3} \pi^{-2r} h^{2r}(2\sqrt{3} r^2 \eta^{-1})^{2r-3}r^2
					\right) h^2   |u|_{\mcH^2(\Omega)}^2
	\end{align*}
	and since $4\sqrt 3 \pi^{-1}\le 3$ further~\eqref{eq:thrm:approx:1}.
	
	\emph{Case 2.} Assume $\eta_0 > 1$. Define
	\[
		W:=\{
				u \in H^1(\Omega) \;:\;
					u\circ G_k \in S_{1,1}(\widehat{\Omega})
					\mbox{ for all } k = 1,\ldots,K
		\},
	\]
	i.e., the set of all globally continuous functions which are locally just linear.
	Observe that $W\subseteq V_h$. Using $u$ and $w$ being continuous, we obtain
	\begin{align*}
		\|u-w\|_{Q_h^+}^2 = |u-w|_{\mc H^1(\Omega)}^2 + \frac{h^2}{\sigma^2 } |u-w|_{\mc H^2(\Omega)}^2.
	\end{align*}
	For the choice
	\[
				w\in H^1(\Omega) \quad \mbox{with}\quad w|_{\Omega_k} := w_k = \widehat{w}_k\circ G_k^{-1},
	\]
	where
	\[
			\widehat{w}_k(x,y) :=
				\sum_{i=0}^1 \sum_{j=0}^1
				\widehat{\phi}_i(x)\phi_j(y) \widehat{u}_k(i,j) \quad \mbox{where}\quad
				\widehat{\phi}_0(t):=1-t \quad \mbox{and}\quad \widehat{\phi}_1(t):=t,
	\]
	we further obtain using
	standard approximation error estimates and
	Lemma~\ref{lem:geoequiv:1}
	\begin{align*}
		\inf_{v_h \in V_h} \|u-v_h\|_{Q_h^+}^2 
		\le \|u-w+c\|_{Q_h^+}^2  =
		\|u-w\|_{Q_h^+}^2
			&\lesssim (1+h^2 \sigma^{-2} )|u|_{\mc H^2(\Omega)}^2,
	\end{align*}
	where $c:=(w,1)_{L_2(\Omega)}/(1,1)_{L_2(\Omega)}$.
	Using $\sigma^{-2} \le p^{-4}\le \tfrac{1}{16}$ and $h\lesssim1$ and
	$\eta_0 > 1$, we have
	\begin{align*}
		\inf_{v_h \in V_h} \|u-v_h\|_{Q_h^+}^2 &\lesssim 
		\left(\eta_0^2+\frac{h^2}{16}\right)
		|u|_{\mcH^2(\Omega)}^2 \le (1+\eta_0^2h^{-2})
		h^2 |u|_{\mcH^2(\Omega)}^2
		,
	\end{align*}
	which shows~\eqref{eq:thrm:approx:1} also for the second case.
	
	Finally, we show that $\Psi$ is such that the desired
	bound~\eqref{eq:thrm:approx} follows. We again consider two cases.
	
	\emph{Case 1.} Assume $\ln(\sigma^{1/2})\le p_{\min}$. In this case, we choose 
	\[
					r := \max \{2, \lceil \ln(\sigma^{1/2}) \rceil \}
	\quad
	\mbox{and}
	\quad
					\eta_0 := 3 \ee r^2 h,
	\]
	where $\ee$ is Euler's number ($\ln\ee=1$), and obtain
	\[
		\Psi   \lesssim  
		 r^4  + \sigma  \ee^{-2r} \eqsim r^4
			\lesssim 
				 \big(\ln \sigma\big)^4
			\lesssim \big(\ln \sigma\big)^4 \sigma^{2/(2p_{\min}-1)},
	\]
	which finishes the proof for Case 1.
	
	\emph{Case 2.} Assume $\ln(\sigma^{1/2})\ge p_{\min}$. In this case, we choose 
	\[
					r := p_{\min}
	\quad
	\mbox{and}
	\quad
					\eta_0 := 3 \sigma^{1/(2r-1)} r^2 h
	\]
	and obtain immediately
	\[
		\Psi = 1 + 3^2 \sigma^{2/(2r-1)} r^4
					+ \sigma^{2/(2r-1)}   r^2 
					 \lesssim \sigma^{2/(2p_{\min}-1)} (\ln\sigma)^4,
	\]
	which finishes the proof for Case 2.
\qed\end{proof}
\begin{proof}[of Theorem~\ref{thrm:high}]
	Let $u\in \mcH^{2,\circ}(\Omega) \cap H^1(\Omega)$ be arbitrary but
	fixed and define $u_h:=\Pi u$.
	The defintion of the $Q_h^+$-norm and the $Q_h$-norm yield
	\[
			\| u-v_h \|_{Q_h^+}^2
			=
			\| (I-\Pi)u \|_{\mathcal H^1}^2
			+
			\frac{\sigma}{h} \sum_{(k,l)\in \mathcal N}  
			\| \llbracket (I-\Pi)u \rrbracket \|_{L_2(I_{k,l})}^2
			+
			\frac{h^2}{\sigma^2}		| (I-\Pi)u |_{\mathcal H^2}^2.
	\]
	Using the triangle inequality and 
	Lemma~\ref{lem:geoequiv:1}, we have further
	\[
			\| u-v_h \|_{Q_h^+}^2
			=
			\sum_{k=1}^K\left(
			\| \widehat w_k \|_{H^1(\widehat \Omega)}^2
			+
			\frac{\sigma}{h} \sum_{(k,l)\in \mathcal N}  
			\| \widehat w_k \|_{L_2(\partial \widehat\Omega)}^2
			+
			\frac{h^2}{\sigma^2}	
			| \widehat w_k |_{H^2(\widehat \Omega)}^2\right).
	\]
	where $\widehat w_k:=(I-\widehat \Pi_k)\widehat u_k$. Using a
	trace theorem~\cite[Lemma~4.4]{Takacs:2018} and
	$AB \le \gamma A^2 + \gamma^{-1} B^2$, we obtain
	\begin{align*}
			\| \widehat w_k \|_{L_2(\partial \widehat\Omega)}^2 
			\lesssim
			 	\| \widehat w_k \|_{L_2(\widehat\Omega)}
			 	\| \widehat w_k \|_{H^1(\widehat\Omega)}
			\le
			 	h^{-1} \| \widehat w_k \|_{L_2(\widehat\Omega)}^2
			 	h \| \widehat w_k \|_{H^1(\widehat\Omega)}^2.
	\end{align*}
	Using this, $h\lesssim1$, and $2\le p^2\lesssim \sigma$, we obtain
	\[
			\| u-v_h \|_{Q_h^+}^2
			\lesssim
			\sum_{k=1}^K\left(
			\frac{\sigma}{h^2}
			\| \widehat w_k \|_{L_2(\widehat \Omega)}^2
			+
			\sigma			
			| \widehat w_k |_{H^1(\widehat \Omega)}^2
			+
			\frac{h^2}{p^2}	
			| \widehat w_k |_{H^2(\widehat \Omega)}^2\right).
	\]
	The desired result is then a consequence
	of Lemma~\ref{lem:h1:h2} and the assumed equivalence of the
	norms on the physical domain and the parameter domain. 
\qed\end{proof}

\section{Numerical Experiments}\label{sec:num}

We depict the results of this paper with numerical results. We choose
a spline approximation of the quarter annulus $\Omega:=\{(x,y)\,:\,
x,y>0,\; 1<x^2+y^2<4\}$. This domain is uniformly split into
$4\times 4$ patches in the obvious way. We solve the Poisson equation
\[
		- \Delta u = 2\,\pi^2 \sin(x\;\pi) \; \sin(y\;\pi) \mbox{ on } \Omega \quad \mbox{and}\quad
		u = u^* \mbox{ on } \partial \Omega,
\]
where
$
		u^*(x,y) = \sin(x\;\pi) \; \sin(y\;\pi)
$
is the exact solution. On each patch, we introduce a coarse discretization
space $V_{h_0}$ for $\ell=0$, which
only consists of global polynomials of degree $p$. We then refine all grids
$\ell=1,2,\ldots$ times uniformly. Next, we modify the discretization
spaces in order to obtain non-matching discretizations at the interfaces
(since fully matching distretizations would be a special case that
would allow a conforming discretization that would not be
of interest in a discontinuous Galerkin setting). The modification
means that we refine the grid one additional time for one third
of the patches and that we increase the spline degree to $p+1$
(and the smoothnes to $C^p$) for another third of the patches. 
In Table~\ref{tab:1} and Figure~\ref{fig:1}, we depict the discretization errors
$e_{\ell,p}$ in the $H^1$-norm relative to that of the solution
and the rates $r_{\ell,p}$, given by
\[
	e_{\ell,p}:=\frac{|u_{h_\ell,p}-u^*|_{\mathcal H^1}}{|u^*|_{\mathcal H^1}}
	\qquad\mbox{and}\qquad
	r_{\ell,p}:=\frac{|u_{h_{\ell-1},p}-u^*|_{\mathcal H^1}}{|u_{h_\ell,p}-u^*|_{\mathcal H^1}}.
\]
 
\newcommand{\tabh}[1]{\multicolumn{2}{c|}{$p=#1$}}
\newcommand{\tabhy}[1]{\multicolumn{3}{c|}{$p=#1$}}
\newcommand{\tabhx}[1]{\multicolumn{3}{c}{$p=#1$}}
\newcommand{\tabe}[1]{#1\hspace{-.6em}}

\begin{table}[ht]
\begin{center}
  \begin{tabular}{l|rlr|lr|lr|rlr}
  \toprule
      &\tabhy{2}&\tabh{3}&\tabh{4}&\tabhx{5}             \\
    $\ell$ & \tabe{dofs} & $e_{\ell,p}$ & $r_{\ell,p}$ & $e_{\ell,p}$ & $r_{\ell,p}$ & $e_{\ell,p}$ & $r_{\ell,p}$ & \tabe{dofs} & $e_{\ell,p}$ & $r_{\ell,p}$  \\ 
    \midrule
   2 & \tabe{   852} & \tabe{ 0.7299} &      & \tabe{ 0.6559} &      & \tabe{ 0.5790} &      & \tabe{  1668} & \tabe{ 0.5228} &       \\
   3 & \tabe{  2756} & \tabe{ 0.3542} &  2.1 & \tabe{ 0.3574} &  1.8 & \tabe{ 0.2827} &  1.4 & \tabe{  4100} & \tabe{ 0.2756} &  1.9  \\
   4 & \tabe{  9828} & \tabe{ 0.0709} &  5.0 & \tabe{ 0.0259} & 13.8 & \tabe{ 0.0102} &  2.0 & \tabe{ 12228} & \tabe{ 0.0042} & 66.2  \\
   5 & \tabe{ 37028} & \tabe{ 0.0141} &  5.0 & \tabe{ 0.0020} & 12.8 & \tabe{3.2e--4} & 33.0 & \tabe{ 41540} & \tabe{5.3e--5} & 78.1  \\
   6 & \tabe{143652} & \tabe{ 0.0033} &  4.3 & \tabe{2.2e--4} &  9.2 & \tabe{1.6e--5} & 20.1 & \tabe{152388} & \tabe{1.2e--6} & 44.4  \\
   7 & \tabe{565796} & \tabe{8.1e--4} &  4.1 & \tabe{2.7e--5} &  8.3 & \tabe{9.4e--7} & 17.1 & \tabe{582980} & \tabe{3.6e--7} & 33.9  \\
   \midrule
   \multicolumn{3}{l}{Expected:} & 4 && 8 && 16 &&& 32  \\
  \bottomrule
  \end{tabular}
\end{center}
\caption{Discretization errors}
\label{tab:1}
\end{table}

The numerical experiments show that the error decreases
like $h_\ell\eqsim 2^{-\ell}$
or even better. Whether or not the error bound depends on
$p^2$ or $\ln p$ cannot be seen in this experiment. The
almost $p$-robust convergence of multigrid solvers whose analysis
follows from the presented results, can be seen in~\cite{Takacs:2019a}.
Since we choose splines of maximum smoothness, the number
of degrees of freedom only grows mildly.

\begin{figure}[th]\centering
\begin{tikzpicture}[scale=.6, transform shape]
	\begin{semilogyaxis}[
	minor x tick num=0,
	minor y tick num=9,
	xmin=2, xmax=7,
	ymin=0.00000001, ymax=1.,
	y=.4cm,
	x=1.5cm,
	grid=major,
	legend style={
	at={(0.15,0.05)},
	anchor=south
	}
	]

	\addplot+ [blue,mark=*] coordinates {
	(2,    0.72986585)
	(3,    0.35420757)
	(4,    0.07093129)
	(5,    0.01410018)
	(6,    0.00330666)
	(7,    0.00081292)
	};% p=2
	\addplot plot [domain=4:7,dashed,blue,mark=none,forget plot] {0.07093129*4^(-x+4)};

	\addplot [red,mark=diamond] coordinates {
	(2,    0.65591065)
	(3,    0.35739814)
	(4,    0.02589521)
	(5,    0.00202052)
	(6,    0.00021976)
	(7,    0.00002650)
	};% p=3
	\addplot plot [domain=4:7,red,dashed,mark=none,forget plot]{0.02589521*8^(-x+4)};

	\addplot [teal,mark=triangle] coordinates {
	(2,    0.57898978)
	(3,    0.28269134)
	(4,    0.01028543)
	(5,    0.00032156)
	(6,    0.00001597)
	(7,    0.00000094)
	};% p=4
	\addplot plot [domain=4:7,teal,dashed,mark=none,forget plot]{0.01028543*16^(-x+4)};

	\addplot [brown,mark=square] coordinates {
	(2,    0.52281262)
	(3,    0.27560317)
	(4,    0.00416378)
	(5,    0.00005334)
	(6,    0.00000120)
	(7,    0.000000036)
	};% p=5
	\addplot plot [domain=4:7,brown,dashed,mark=none,forget plot]{0.00416378*32^(-x+4)};

	\legend{$p=2$\\
	$p=3$\\
	$p=4$\\
	$p=5$\\
	}
	\end{semilogyaxis}\end{tikzpicture}
	\caption{Discretization errors}\label{fig:1}
\end{figure}

\textbf{Acknowledgments.} The author was supported by the Austrian Science Fund (FWF):
S117 and P31048 and by the bilateral project DNTS-Austria 01/3/2017
(WTZ BG 03/2017), funded by Bulgarian National Science Fund and OeAD (Austria).

\bibliographystyle{amsplain}

\begin{thebibliography}{10}

\bibitem{Arnold:1982}
D.~Arnold, \emph{An interior penalty finite element method with discontinuous
  elements}, SIAM J. Numer. Anal. \textbf{19} (1982), no.~4, 742 -- 760.

\bibitem{Arnold:Brezzi:Cockburn:Marini:2002}
D.~Arnold, F.~Brezzi, B.~Cockburn, and L.~Marini, \emph{Unified analysis of
  discontinuous {G}alerkin methods for elliptic problems}, SIAM J. Numer. Anal.
  \textbf{39} (2002), no.~5, 1749 -- 1779.

\bibitem{Bazilevs:BeiraoDaVeiga:Cottrell:Hughes:Sangalli:2006}
Y.~Bazilevs, L.~Beir{\~a}o da~Veiga, J.~A. Cottrell, T.~J.~R. Hughes, and
  G.~Sangalli, \emph{Isogeometric analysis: approximation, stability and error
  estimates for h-refined meshes}, Math. Models Methods Appl. Sci. \textbf{16}
  (2006), no.~07, 1031 -- 1090.

\bibitem{BeiraoDaVeiga:Buffa:Rivas:Sangalli:2011}
L.~Beir{\~a}o da~Veiga, A.~Buffa, J.~Rivas, and G.~Sangalli, \emph{Some
  estimates for {$h$}-{$p$}-{$k$}-refinement in isogeometric analysis}, Numer.
  Math. \textbf{118} (2011), no.~2, 271 -- 305.

\bibitem{Dauge:1988}
M.~Dauge, \emph{{Elliptic boundary value problems on corner domains. Smoothness
  and asymptotics of solutions}}, Lecture Notes in Mathematics, 1341. Springer,
  1988.

\bibitem{Dauge:1992}
\bysame, \emph{{Neumann and mixed problems on curvilinear polyhedra}}, Integr.
  Equat. Operat. Theor. \textbf{15} (1992), 227 -- 261.

\bibitem{Floater:Sande:2017}
M.~Floater and E.~Sande, \emph{{Optimal spline spaces of higher degree for
  $L_2$ $n$-widths}}, J. Approx. Theor. \textbf{216} (2017), 1 -- 15.

\bibitem{Hofreither:Takacs:2017}
C.~Hofreither and S.~Takacs, \emph{Robust multigrid for isogeometric analysis
  based on stable splittings of spline spaces}, SIAM J. Numer. Anal. \textbf{4}
  (2017), no.~55, 2004 -- 2024.

\bibitem{Hofreither:Takacs:Zulehner:2017}
C.~Hofreither, S.~Takacs, and W.~Zulehner, \emph{A robust multigrid method for
  isogeometric analysis in two dimensions using boundary correction}, Comput.
  Methods Appl. Mech. Eng. \textbf{316} (2017), 22 -- 42.

\bibitem{Hughes:2005}
T.~J.~R. Hughes, J.~A. Cottrell, and Y.~Bazilevs, \emph{Isogeometric analysis:
  {CAD}, finite elements, {NURBS}, exact geometry and mesh refinement}, Comput.
  Methods Appl. Mech. Eng. \textbf{194} (2005), no.~39 -- 41, 4135 -- 4195.

\bibitem{Kapl:Sangalli:Takacs:2018}
M.~Kapl, G.~Sangalli, and T.~Takacs, \emph{Isogeometric analysis with {$C^1$}
  functions on unstructured quadrilateral meshes}, SMAI journal of
  computational mathematics \textbf{S5} (2019), 67--86.

\bibitem{LMMT:2015}
U.~Langer, A.~Mantzaflaris, S.~Moore, and I.~Toulopoulos, \emph{{Multipatch
  discontinuous Galerkin Isogeometric Analysis}}, {Isogeometric Analysis and
  Applications 2014} (J{\"u}ttler and Simeon, eds.), {Springer}, 2015, pp.~1 --
  32.

\bibitem{LT:2015}
U.~Langer and I.~Toulopoulos, \emph{Analysis of multipatch discontinuous
  galerkin iga approximations to elliptic boundary value problems}, Comput.
  Vis. Sci. \textbf{17} (2015), no.~5, 217 -- 233.

\bibitem{Necas:1967}
J.~Necas, \emph{{Les m\'ethodes directes en th\'eorie des \'equations
  elliptiques}}, Masson, Paris, 1967.

\bibitem{Petzoldt:2001}
M.~Petzoldt, \emph{{Regularity and error estimators for elliptic problems with
  discontinuous coefficients}}, Ph.D. thesis, {Freie Universit\"at Berlin,
  Weierstra\ss{}–Institut f\"ur Angewandte Analysis und Stochastik}, 2001.

\bibitem{Petzoldt:2002}
\bysame, \emph{A posteriori error estimators for elliptic equations with
  discontinuous coefficients}, Adv. Comput. Math. \textbf{16} (2002), no.~1, 47
  -- 75.

\bibitem{Riviere:2008}
B.~Rivière, \emph{Discontinuous {Galerkin} methods for solving elliptic and
  parabolic equations}, Society for Industrial and Applied Mathematics, 2008.

\bibitem{Sande:Manni:Speleers:2018}
E.~Sande, C.~Manni, and H.~Speleers, \emph{Sharp error estimates for spline
  approximation: explicit constants, $n$-widths, and eigenfunction
  convergence}, Math. Models Methods Appl. Sci. \textbf{29} (2018), no.~6, 1175
  -- 1205.

\bibitem{SchneckenleitnerTakacs:2020}
R.~Schneckenleitner and S.~Takacs, \emph{{Convergence theory for IETI-DP
  solvers for discontinuous Galerkin Isogeometric Analysis that is explicit in
  $h$ and $p$}}, Tech. report, 2020, arXiv: 2005.09546.

\bibitem{Schwab:1998}
C.~Schwab, \emph{$p$- and $hp$-finite element methods: Theory and applications
  in solid and fluid mechanics}, Numerical Mathematics and Scientific
  Computation, Clarendon Press, Oxford, 1998.

\bibitem{Takacs:2018}
S.~Takacs, \emph{Robust approximation error estimates and multigrid solvers for
  isogeometric multi-patch discretizations}, Math. Models Methods Appl. Sci.
  \textbf{28} (2018), no.~10, 1899 -- 1928.

\bibitem{Takacs:2019a}
\bysame, \emph{{Fast multigrid solvers for conforming and non-conforming
  multi-patch Isogeometric Analysis}}, Comput. Methods Appl. Mech. Eng.
  \textbf{371} (2020), no.~113301.

\bibitem{Takacs:Takacs:2015}
S.~Takacs and T.~Takacs, \emph{Approximation error estimates and inverse
  inequalities for {B}-splines of maximum smoothness}, Math. Models Methods
  Appl. Sci. \textbf{26} (2016), no.~07, 1411 -- 1445.

\end{thebibliography}
\providecommand{\bysame}{\leavevmode\hbox to3em{\hrulefill}\thinspace}
\providecommand{\MR}{\relax\ifhmode\unskip\space\fi MR }
% \MRhref is called by the amsart/book/proc definition of \MR.
\providecommand{\MRhref}[2]{%
  \href{http://www.ams.org/mathscinet-getitem?mr=#1}{#2}
}
\providecommand{\href}[2]{#2}

\end{document}